\pdfoutput=1
\documentclass[paper=letter, fontsize=10pt,leqno]{scrartcl}
\usepackage[body={5.35in,7.5in}]{geometry} 

\usepackage[english]{babel}				
\usepackage{amsmath,tabu}				
\usepackage{amsfonts}
\usepackage{amsthm}
\usepackage{latexsym}
\usepackage{hyperref}

\usepackage[pdftex]{graphicx}				
\usepackage{url}
\usepackage{amssymb}
\usepackage{bbm}
\usepackage{MnSymbol}

\usepackage[pdftex]{color}
\usepackage{eucal}

\usepackage{amscd}
\usepackage{calrsfs}
\usepackage[format=plain]{caption}
\usepackage{wrapfig}
 \usepackage{float}

\usepackage{sectsty}					
\allsectionsfont{\centering \normalfont\scshape}	

\usepackage{mathtools}
\DeclareCaptionStyle{ruled}{labelfont=bf,labelsep=space}
\usepackage[]{forest}

\newtheorem{theorem}{Theorem}

\newtheorem{definition}[theorem]{Definition}
\newtheorem{proposition}[theorem]{Proposition}

\newtheorem{corollary}[theorem]{Corollary}

\newtheorem*{proposition*}{Proposition}

\title{
		\vspace{-1in} 	
		\usefont{OT1}{bch}{b}{n}
		\normalfont \normalsize \textsc{} \\ [25pt]
		\huge  Rolling and no-slip bouncing in cylinders
}

\date{}

\author{\normalfont \large 
 T. Chumley\footnote{\scriptsize Department of Mathematics and Statistics, Mount Holyoke College, 50 College St, South Hadley, MA 01075},
\ S. Cook\footnote{\scriptsize Department of Mathematics, Tarleton State University, Box T-0470, Stephenville, TX 76401},
\ C. Cox\footnotemark[2],
\  R. Feres\footnote{\scriptsize Department of Mathematics and Statistics, Washington University, Campus Box 1146, St. Louis, MO 63130}
}

 \date{\today}

\begin{document}

\maketitle

\begin{abstract}
\begin{center}
 Abstract \end{center}
{\small  
The purpose of this paper is to compare a  classical non-holonomic  system\----a sphere rolling against the inner surface of a  vertical cylinder under gravity\----and 
a class of discrete dynamical systems   known as {\em no-slip billiards} in  similar configurations. A well-known notable feature of the non-holonomic system is that the rolling sphere does not fall; its height function is bounded and oscillates harmonically up and down.
The central issue of the present work is whether similar bounded behavior can be observed in the no-slip billiard counterpart. Our  main results are  as follows: for circular cylinders in dimension $3$, the no-slip billiard has the bounded orbits property, and very closely approximates rolling motion, for a class of initial conditions which we call {\em transversal rolling impact}. When this condition does not hold,
trajectories undergo vertical oscillations superimposed to an overall downward acceleration. Considering cylinders with different cross-section shapes, we show that no-slip billiards between two parallel hyperplanes in Euclidean space of arbitrary dimension are always  bounded even under a constant force parallel to the plates; for general cylinders, when the orbit of the  transverse system (a concept that depends on a factorization of the  motion into transversal and longitudinal components) has period two\----a very common occurrence in planar no-slip billiards\----the motion in the longitudinal direction, under no forces, is generically not bounded. This is shown using a formula for a  longitudinal linear drift that we prove in arbitrary dimensions. While the systems for which we can prove 
the existence of bounded orbits have relatively simple transverse dynamics, we also briefly explore numerically a no-slip billiard system, namely the stadium cylinder billiard,
that can exhibit chaotic transversal dynamics. 
}
\end{abstract}

\section{Introduction}
A classical example of a non-holonomic mechanical system consists of a ball that rolls against the inner side of a vertical cylinder with enough speed so as not  to lose contact with the surface. We imagine that 
the surface of the ball is ideally rough, or rubbery, so that a kind of conservative static friction causes it to roll without slipping. 
 Contrary to common intuition, the ball does not fall to the ground, but oscillates harmonically up and down. This ideal behavior is
  approximately reproduced in lab experiments; see, for example, \cite{gualtieri}.  (For a study of this system in the context of the theory of nonholonomic systems see also 
 \cite{borisov} and \cite{Neumark}.  We give below a fairly complete description of it in a form that will serve our present needs.)

Rather than rolling,  we imagine that the idealized ball 
bounces off the cylinder's inner wall in a kind of grazing motion, but still under the same kind of  conservative static friction constraint that   couples 
the  linear and rotational components of the motion. Would the ball still defy gravity, so to speak, as in the rolling process? 

This question introduces a number of issues. First, how should we define such a discrete system in a way that is reasonably well-motivated, starting from general physical principles such as energy conservation and time reversibility?  A good candidate   has been studied before under the name of {\em no-slip billiard} systems. For some early papers, see \cite{Garwin,gutkin,W}.  See also \cite{LLM,MLL},
as an example of how no-slip collision models are  used in situations where exchange of linear and angular momentum of particles is desired.
Recently, the authors have begun to  pursue   a more systematic  investigation of the dynamics of no-slip billiards; see \cite{CF,CFII,CFZ}. Except for  \cite{CF}\----a differential geometric study of
rigid collisions in $\mathbb{R}^n$ in which no-slip billiards arise in a  natural way\----we are only aware of studies that are restricted to dimension $2$. Clearly, we need here to consider such systems in dimension $3$ (or greater, if one is also interested, as we are, in related problems of a more differential geometric 
flavor.) 

Another issue to consider is  whether any particular choice of  dynamical system  defined by  sequences of impacts can be said to serve as a good model for a theory of discrete nonholonomic dynamics.  Concretely, does the discrete model exhibit similar properties as the continuous time system, such as having bounded (e.g., not falling to the ground) trajectories; and do grazing trajectories approximate the well-studied rolling motion? 

The purpose of this paper is twofold. First, we wish to pursue  the problem of existence of bounded trajectories of no-slip billiards in generalized cylinders in dimensions $2$ and greater. By a {\em cylinder} we mean a domain in $\mathbb{R}^n$ that has translation symmetry along an axis. (The $3$-dimensional circular cylinder will have some prominence here, but we also consider other domains.) In dimension $2$, \cite{gutkin} showed that the no-slip billiard motion in an infinite strip
is bounded, and in \cite{CFII,CFZ} we extended and refined this observation  in a way that provides some insight into the dynamics of general  polygonal no-slip billiards.  As might be expected, the higher dimensional story is more subtle; we describe in this paper some of the new phenomena  that arise beyond dimension $2$. Another goal is to make a direct comparison between the rolling and
no-slip (bouncing) dynamics.  One of our main results here, which is specific to circular cylinders in dimension $3$, is that  for a set of initial conditions that we refer to as {\em transversal rolling impact}, the billiard motion is indeed bounded. We also show numerically that the discrete system very closely approximates rolling  under the just mentioned class of initial conditions. 
On the other hand, if these  conditions are not satisfied, it is observed numerically that the ball will acquire an overall acceleration under an external force.

There are many  natural questions that we do not yet pursue in this study. For example, when can the discrete system be said to be integrable? (See \cite{borisov} for a proof of boundedness of orbits for rolling in $3$-dimensional cylinders of arbitrary cross-section using the existence of  integrals of motion; our proof that no-slip billiard motion between parallel hyperplanes in $\mathbb{R}^n$ under constant force is bounded also makes used of conserved quantities in a suggestive way.) What can be said   about the preservation of the canonical Liouville measure? (We give sufficient conditions for the invariance of this measure   in 
\cite{CF}, and show that the measure is preserved for no-slip billiards  in dimension $2$, but the question is as yet open in higher dimensions. The non-preservation of the Liouville measure is a feature of many non-holonomic mechanical systems.) 
Moreover, proving that the  differential equations for rolling can be obtained as a limit of the 
discrete equations of the no-slip billiard with short intercollision flights is a natural question that has eluded us so far.  Our numerical experiments suggest that a certain {\em transversal rolling impact defect} parameter introduced below should play a central role. The focus of the present paper is not on  such general issues but is aimed  mainly  at pointing out  interesting phenomena that can be observed when comparing rolling and bouncing motion in such velocity constrained systems, as well as proving results on boundedness of orbits for generalized cylinders that do not have a rolling counterpart. 

The rest of the paper is outlined as follows.  The next section summarizes the paper's main observations.  Section \ref{sec:more definitions and basic facts} sets up the geometric notation and background for defining the no-slip collision map.  Section \ref{sec:rolling motion on cylinders} gives a self-contained overview of rolling on cylinders  with general cross section.  Sections \ref{sec:No-slip billiards in  general cylinders}, \ref{sec:Forced billiard motion in a circular cylinder}, and \ref{sec:Forced motion between parallel planes} give proofs of the main results.  We conclude in Section \ref{sec:final comments} with some numerical observations and a proposal for future work on no-slip billiards in cylinders whose transverse dynamics exhibit chaotic behavior.

 \section{Main definitions and results} We begin with a few definitions and, in particular, recall the notion  of no-slip billiards.  
 Let $D=D_0(r)$ denote the ball of radius $r$ centered at the origin in $\mathbb{R}^n$. It is given a mass distribution measure $\mu$ having total mass $m$. We assume that the first moment $\overline{x}=\int_Dx\, d\mu(x)$ is zero and let
$L=(l_{ij})$ be the matrix of second moments per unit  mass:
$$l_{ij}:=\frac1m \int_{D} x_ix_j\, d\mu(x). $$
Only  mass distributions for which $L$ is scalar, $L=\lambda I$, where $I$ is the identity matrix in dimension $n$, will be considered here. When $\mu$ is rotationally symmetric and has a continuous  density   relative to the volume measure, then, expressed in terms of the density 
as a function of the radial coordinate, 
$$ \lambda=\frac{V_n}{m}\int_0^r s^{n+1} \rho(s)\, ds$$
where $V_n$ is the volume of an $n$-dimensional ball of radius $1$.  
It is convenient to introduce the parameter $\gamma=\sqrt{2\lambda}/r$.
For  example, $\lambda= \frac{r^2}{n+2}$ and $\gamma=\sqrt{\frac{2}{n+2}}$ if $\rho=m/V_n r^n$ is constant.  It is not difficult to see that, for a general rotationally symmetric mass distribution, 
$0\leq \lambda\leq r^2/n$ and $0\leq \gamma \leq \sqrt{\frac2{n}}$.

Let $\mathcal{B}_0$ be an open connected region in $\mathbb{R}^n$. Let  $\mathcal{B}$ be the closure of the set of $x\in \mathbb{R}^n$ for which 
the ball   of center $x$ and radius $r$ is contained in $\mathcal{B}_0$. We assume that $\mathcal{B}$ is a
  manifold with corners as defined in \cite{lee}.  We refer to $\mathcal{B}$ as the {\em billiard domain} and, occasionally, to $\mathcal{B}_0$ as the {\em enlarged billiard domain}. For regular points  $a$ of the boundary $\partial \mathcal{B}$ (that is, a point at which a tangent space is defined) define $\nu(a) := \nu_a$ to be the unit normal vector field pointing towards the interior of $\mathcal{B}$. 
  Denote by $SE(n)$ the Euclidean group of positive isometries of $\mathbb{R}^n$. Its
   elements will be written as pairs $(A,a)\in SO(n)\times \mathbb{R}^n$ acting on $\mathbb{R}^n$  by affine transformations $x\mapsto (A,a)x:=Ax+a$ on $\mathbb{R}^n$.

  The configuration manifold of a spherical particle of radius $r$ with center in $\mathcal{B}$ is the manifold with corners $M\subset SE(n)$ consisting of all $(A,a)$ such that $a\in \mathcal{B}$.  Naturally, a boundary point of $M$ is a pair $(A,a)$ such that $a\in \partial \mathcal{B}$.

  In this paper we are interested in rolling and bouncing motion in cylinders, by which we will mean the following, unless further geometric assumptions are made:
 \begin{definition}[Cylinders in $\mathbb{R}^n$]\label{cylinder}
  A {\em solid cylinder} in $\mathbb{R}^n$ with {\em axis  vector} $e$ is a domain $\mathcal{B}$ such that $a+se\in \mathcal{B}$ for
  all $a\in \mathcal{B}$ and $s\in \mathbb{R}$. Here $e$ is a unit vector in $\mathbb{R}^n$ that defines which direction is up.
  Writing $\overline{\mathcal{B}}=\mathcal{B}\cap e^\perp$, we have
  $\mathcal{B} = \overline{\mathcal{B}}\times \mathbb{R}e$. The boundary cylinder will be written $S=\overline{S}\times \mathbb{R}e$, where
  $S=\partial \mathcal{B}$.
  \end{definition}

 As far as the rolling process is concerned,  all we will  need is the boundary cylinder $S$; the solid cylinder will be needed for the  no-slip   billiard systems.  For  rolling  we also assume that
 $S$ is a smooth hypersurface in $\mathbb{R}^n$, whereas for the billiard motion we may allow singular points so long as we consider orbits that avoid them.

 From the spherically symmetric mass distribution $\mu$ with second moments matrix $L=\lambda I$, $\lambda = \frac{1}{2}(r\gamma)^2$, we define the {\em kinetic energy Riemannian metric} on $M$  as follows: Let $\xi=(U_\xi A, u_\xi)$ and $\eta=(U_\eta A, u_\eta)$ be vectors tangent to $M$ at $(A,a)$. Then 
 \begin{equation}\label{metric}\left\langle \xi,\eta\right\rangle :=m\left\{\frac{(r\gamma)^2}{2}\text{Tr}\left(U_\xi U_\eta^\dagger\right) + u_\xi\cdot u_\eta\right\}. \end{equation}
  It is  easily checked that this bilinear form is (the restriction to $TM$ of) a left-invariant Riemannian metric on  $SE(n)$.

  Observe that if $(A(t),a(t))$ is a differentiable curve in $M$, then its derivative  at $t=0$  is given by $(\dot{A},\dot{a})$ where  $\dot{a}=u\in \mathbb{R}^n$ and $\dot{A}=UA$, where  $U\in  \mathfrak{so}(n)$ is an element of the Lie algebra of the rotation group. The kinetic energy of the moving ball with state $\xi$ at  configuration $(A,a)$ is then written in the norm associated to the Riemannian metric as $\frac12\|\xi\|^2$. Notice    that the metric and the kinetic energy function do not depend on $A$.

    The boundary of the configuration manifold $M$ has a special structure that will be important for our concerns, which we call the {\em no-slip bundle}. It is a vector subbundle of the tangent bundle to $\partial M$, denoted $\mathfrak{S}$, and defined as follows.
    \begin{definition}[The no-slip bundle]\label{no slip bundle}
    At each regular point  $q=(A,a)\in SE(n)$, $a\in S=\partial \mathcal{B}$, describing a boundary configuration of the moving particle   system, define
    \begin{equation}\label{S}\mathfrak{S}_q=\left\{(UA,u)\in T_qM: u=rU \nu_a \right\}\end{equation}
    where  $r$ is the radius of the particle. We call $\mathfrak{S}_q$ the {\em no-slip space} at $q$.
    \end{definition}
    
    This definition has a clear motivation if we note that the point of contact  $x= a-r\nu_a$ of the moving particle with the boundary of the extended domain  $\mathcal{B}_0$
   has velocity
    $$ v_x = U(x-a)+u = -rU\nu_a+ rU\nu_a = 0.$$
  Thus a state of the system (that is, a   tangent vector at some point of $M$) at a boundary configuration $q$ is in the no-slip bundle if the point of contact of the particle with the boundary of the domain $\mathcal{B}_0$  has $0$ velocity. (The reader should keep in mind the distinction between the billiard domain $\mathcal{B}$ and the enlarged domain   $\mathcal{B}_0$; the former contains the center of masses of the particle, and the latter contains the entire ball of radius $r$ around those centers in $\mathcal{B}$.)

    \begin{definition}[No-slip rolling] A particle whose
  motion is described by  a smooth curve $q(t)\in \partial M$ is said to undergo {\em (no-slip) rolling} if $\dot{q}(t)\in \mathfrak{S}_{q(t)}$ for each $t$. 
 We also say in this case that the particle {\em rolls} with no slip on the boundary surface of the (enlarged) billiard domain. 
    \end{definition}

  We consider next the dynamical equations describing the   motion of a particle of radius $r$, mass $m$, and rotationally symmetric mass distribution
  with parameter $\gamma$, that rolls with no-slip on the boundary of the enlarged domain $\mathcal{B}_0$. With some innocuous abuse of language we speak of rolling on $S=\partial \mathcal{B}$, the locus of the centers of mass of the moving particle, rather than the hypersurface of   contact points. Keeping this in mind, we will avoid when possible  referring to  $\mathcal{B}_0$.

  For the details on non-holonomically constrained systems we suggest \cite{bloch}. Let $f$ be a vector field on $M$ which we interpret as a force field. 
  The motion of the unconstrained system is governed by Newton's equation $m\frac{\nabla \dot{q}}{dt}=f$, where $\nabla$ is the Levi-Civita connection for
  the kinetic energy Riemannian metric (\ref{metric}). The constraint (rolling on the boundary hypersurface $S$) can be imposed by adding a force field $N$ taking values in $\mathfrak{S}^\perp$. This is the force needed to keep $\dot{q}$ in $\mathfrak{S}_q$. As $\left\langle N, \dot{q}\right\rangle_q=0$, the constraint force  does no work. 
 \begin{definition}[Constrained Newton's equation]\label{Newton}
A smooth path $q(t)\in SO(n)\times S$ satisfies the {\em constrained Newton's equation} with force field $f$  and non-holonomic constraint defined by the no-slip bundle $\mathfrak{S}$ if $\dot{q}(t)\in \mathfrak{S}_{q(t)}$ for all $t$ and
$$\frac{\nabla \dot{q}}{dt}  = m^{-1}f +N $$
where $N=N(q,\dot{q})$ lies in  $\mathfrak{S}^\perp$. 
 \end{definition}

  So far the interior of $\mathcal{B}$ has not been relevant since in the rolling process the center of the moving particle must remain on the
  boundary hypersurface $S$. This will change, of course, as 
  we consider next the notion of {\em no-slip bouncing}. Before stating the definition, we recall the set-up of no-slip billiard systems.
  (The reader is referred to \cite{CF} for a detailed account of what is briefly skimmed over below, and to \cite{CFZ} for some dynamical results for $2$-dimensional systems.) In a billiard system in $\mathcal{B}$, 
  the motion  of the  particle (of radius $r$ and  spherically symmetric mass distribution of total mass $m$ and distribution parameter $\gamma$) 
  consists of a sequence of flight segments in the interior of $M\subset SE(n)$ separated by instantaneous  collisions with the boundary $\partial M$. 
  The inter-collision flights are described by the unconstrained Newton's equation $m\frac{\nabla \dot{q}}{dt}=f$, or by geodesic motion  $\frac{\nabla \dot{q}}{dt}=0$ if no forces are present. Collisions at a given point $q\in \partial M$ are defined by a linear map $C_q:T_qM\rightarrow  T_qM$ that sends   {\em pre-collision}
  vectors, that is, vectors in $\{v\in T_qM: \langle v,\mathbbm{n}_q\rangle\leq 0\}$ to   {\em post-collision} vectors in $\{v\in T_qM: \langle v,\mathbbm{n}_q\rangle\geq 0\}$, where $\mathbbm{n}_q$ is the inward pointing unit vector at a boundary point $q\in M$.

The map $C_q$ will be selected based on the following physical assumptions: (1) Collision is   energy preserving; this means that  $C_q$ is an orthogonal map relative to the Riemannian norm on $M$. (2) It is time reversible, which forces $C_q$ to be a linear involution. (3) The orthogonal component of the pre-collision velocity  in the no-slip subspace $\mathfrak{S}_q$ is not affected by $C_q$. This third requirement is natural since an (instantaneously) rolling collision,
for which the point of contact is stationary,  
  should not cause a change in linear or angular velocities due to an  exchange of momentum at impact, except in the direction $\mathbbm{n}$. In fact, a more careful analysis of the impact event (as in \cite{CF}) identifies the orthogonal complement of $\mathfrak{S}_q$ as the space containing  {\em impulse} vectors. 
  Ordinary billiard systems are those for which $C_q$ is   the identity not only on $\mathfrak{S}_q$ but also on the subspace of $\mathfrak{S}^\perp_q$ perpendicular to $\mathbbm{n}_q$, in which case rotation and translation components of the motion are de-coupled and the rotation part may be ignored. 
  In other words, for standard billiards, the collision maps $C_q$ are specular reflection (in the kinetic energy inner product).
  This corresponds to a perfectly slippery contact between the moving particle and the boundary of the billiard domain. 
   If instead we wish to model the behavior of a perfectly elastic ball with a perfectly rough (or rubbery) surface, for which angular and linear velocities may be partly exchanged at collision, then we are forced under the other  assumptions to require $C_q$ to be the negative of the identity map on the full orthogonal 
   complement of the no-slip subspace (which here includes $\mathbbm{n}_q$.)
 With this in mind we state the following definition.
 
 \begin{definition}[No-slip bounce]\label{Cq defined}
 A {\em no-slip} bounce at $q\in \partial M$ is the correspondence $v^-\mapsto v^+$  sending a pre- to a post-collision velocity at $q$ defined by the no-slip collision map $C_q$. The latter is the orthogonal (relative to the  kinetic energy Riemannian metric (\ref{metric})) linear involution of $T_qM$ equal to the
 identity on $\mathfrak{S}_q$ and minus the identity on this space's full orthogonal complement. 
 \end{definition}

  Due to the spherical symmetry of the moving particle's mass distribution and the discrete nature of the billiard process, it is less relevant for no-slip billiard systems that we keep track of the actual rotation matrix $A$. In fact, prior to each collision, we can imagine that the particle is rotated back to a fixed reference orientation in space, keeping linear and angular velocities unaffected.
This leads to 
  the notion of  {\em reduced phase space} $\mathcal{N}$   at boundary configurations. With  $S=\partial\mathcal{B}$, we define 
  \begin{equation}\label{reduced phase space}\mathcal{N}:= \left.T\mathcal{B}\right|_{S} \times \mathfrak{so}(n)\end{equation}
  whose elements are triples $(a, u, U)$,  $u\in T_a\mathcal{B}$.  
  In other words, we may assume that, at each collision, the rotation matrix is the identity and only keep track of the matrix $U\in \mathfrak{so}(n)$
  of infinitesimal angular velocities, in addition to the point $(a,u)$ representing the velocity of the center of mass  at  $a\in S$.  
  Since $A$ will not be involved, we write $C_a$ rather than $C_q$ when viewing it as a map on $\mathcal{N}_a:=T_a\mathcal{B}\times \mathfrak{so}(n).$

 We now turn to no-slip billiards on solid cylinders $\mathcal{B}$ with axis vector $e$.   If $a$ is a point in $\mathcal{B}$, we occasionally write $\bar{a}$ for its projection to the cross-section $\overline{\mathcal{B}}$; similarly, if $u$ is a tangent vector at  $a$, its projection to a vector  at $\bar{a}$ on the cross-section may be written $\bar{u}$.

The cross-section domain $\overline{\mathcal{B}}$ has its own no-slip billiard system, with a ball of dimension $n-1$ as the moving particle. The latter will also have a spherically symmetric mass distribution (if this is the case  for the $n$-dimensional particle) given by the   marginal mass distribution after integrating   along   $e$. In this case the parameter $\gamma$ for the $(n-1)$-dimensional particle is the same as for the $n$-dimensional one.

For standard billiard systems on a cylinder (for which reflection is specular), it is both clear and unremarkable that trajectories of the $n$-dimensional system should project to trajectories of the billiard system on $\overline{\mathcal{B}}$. This is due to conservation of momentum resulting from the translation symmetry along $e$. For the no-slip billiard, this component of momentum is no longer conserved.  Nevertheless
this projection property  still holds. 

\begin{theorem}\label{first theorem}
Let  $\mathcal{N}$ be the reduced phase space of the no-slip billiard system on the solid cylinder domain $\mathcal{B}\subset \mathbb{R}^n$, and let
$\overline{\mathcal{N}}$ be the reduced phase space for the associated  transverse billiard system. Then trajectories of the no-slip billiard on $\mathcal{N}$, possibly with a constant force in the longitudinal direction,  project to trajectories of the no-slip billiard map on $\overline{\mathcal{N}}$, where the latter system is given the same mass distribution parameter $\gamma$ as the billiard in dimension $n$. 
 \end{theorem}

 \begin{figure}[htbp]
\begin{center}
\includegraphics[width=3.0in]{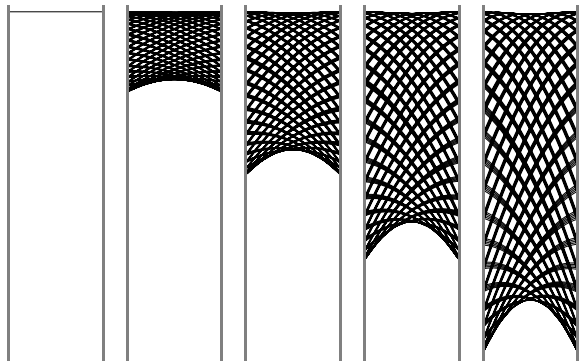}\ \ 
\caption{{\small  No-slip billiard system between two parallel plates under gravity have bounded orbits in all dimensions. Far left: a simple periodic orbit with gravity turned off.  Orbits shown from left to right are under the influence of increasing force. }}
\label{plates}
\end{center}
\end{figure}

 The above theorem only expresses part of 
 a very useful decoupling of the full set of linear and angular velocity components into two that will greatly facilitate the study of  the displacement of the particle's center of mass 
 along the longitudinal direction. (See Proposition \ref{longitudinal equations}.) To anticipate what is to come, we note that the Lie algebra $\mathfrak{se}(n)$ of the Euclidean group $SE(n)$ splits into the Lie algebra $\mathfrak{se}(n-1)$ of $SE(n-1)$, the group associated with  the transversal system, and the orthogonal complement $\mathfrak{m}$ of the latter Lie algebra with respect to the kinetic energy metric. The no-slip  collision map will factorize according to this orthogonal decomposition, and our main interest will be on the factor $\mathfrak{m}$ since this is the one that contains the $e$-component of the center of mass velocity. The subspace $\mathfrak{m}$ has dimension $n$; this means that, so long as we are concerned with the longitudinal motion, we only need to focus on $n$ out of $\frac{n(n+1)}{2}$ velocity components (the latter number being the dimension of $SE(n)$).  
 Therefore, an investigation  of the motion of the   moving particle   naturally splits into a study of the {\em transverse billiard} system on $\overline{\mathcal{B}}$ and a reduced system containing the component of the motion along $e$.

 For $3$-dimensional no-slip billiards in cylinders, this theorem  says that the transversal part of the motion reduces to understanding no-slip billiard dynamics
 in dimension $2$. In the course of this paper we refer to a few  results in  dimension $2$ established in our \cite{CFZ}. In general however
 the focus of the rest of the paper is on the motion along the axis of translation symmetry set by the axis vector $e$, rather than on the transverse  dynamic, and
 in particular on the question of boundedness of orbits. 
Note that we will refer to the direction along $e$ as the longitudinal direction or the {\em vertical} direction, using the terms interchangeably.

We now summarize our main results  concerning the question of whether or not trajectories are bounded in the longitudinal direction of the cylinder. As will be seen, the general answer is: sometimes yes, and sometimes no.  The following theorem extends the main result of \cite{gutkin} from free motion in $2$-dimensional strips to possibly forced motion in  $n$-dimensional regions bounded by parallel hyperplanes. (This domain is a cylinder, according to our general definition.)

 \begin{theorem}\label{two plates bounded}
 Consider a domain whose boundary consists of two parallel hyperplanes in $\mathbb{R}^n$, $n\geq 2$. Then a trajectory of the no-slip billiard system 
 whose initial center of mass velocity is not parallel to the hyperplanes 
  is bounded. Trajectories remain bounded if a constant force   is applied to the particle's center of mass along   any direction parallel to those hyperplanes.
  \end{theorem}

Next we ask wether the property of having bounded orbit holds for general cylinders. We show that even in the absence of forces the motion may not be bounded. Typically the  height of particle's center of mass will  possess  a drift in one direction along the cylinder axis $e$ superimposed to an oscillatory part. 

\begin{wrapfigure}{L}{0.5\textwidth}
\centering
\includegraphics[width=2.5in]{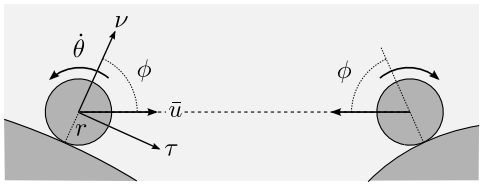}\ \ 
\caption{\label{period2}{\small  Conditions for a transversal period $2$ orbit in dimension $3$.  The projection to $\mathbb{R}^2$ of the velocity $u$ of the center of mass and the angular velocity
$\dot{\theta}$ are related by $|\dot{\theta}|=(m r/\mathcal{J})|u\sin\phi|$, where $m$ is the projected disc's mass, $r$ is its
radius, and $\mathcal{I}$ is its moment of inertia for the projected (or marginal) mass distribution.}}
\end{wrapfigure} 

Before validating this claim,
we  observe that,
due to Theorem \ref{first theorem}, it makes sense to talk about {\em transversely periodic} trajectories of the cylindrical billiard.  
 In dimension $3$, transversal period $2$ orbits
 are very ubiquitous and easy to obtain. (See Section 3 of \cite{CFZ}.) Figure \ref{period2} shows the conditions under which  they arise.
 Observe that, for $n=3$, the constraint equation on initial conditions for period $2$ orbits involves only the projection  of the linear velocity
 to the cross-section plane and 
  of the $e$-component  of the angular velocity vector $\omega$ (which  is  $\dot{\theta}$ as  indicated in Figure
  \ref{period2}). In particular, no conditions are imposed on the velocity components that appear in the description of the longitudinal motion.

What follows is a corollary of Theorem \ref{drift}. (The theorem holds for general  cylinders in arbitrary finite dimension.)
Referring to Figure \ref{period2}, we regard  $\tau$ and $\nu$ as the tangent and normal vectors at the first collision point and $\omega$ as the pre-collision angular velocity at that point; $\phi$ is as indicated in that figure.

 \begin{corollary}\label{drift example} A transversal period $2$  orbit of a $3$-dimensional general cylinder billiard, under no forces, with initial linear vertical velocity $\sigma_0$   and initial angular velocity vector $\omega$ has a vertical drift 
  $$ \lim_{\ell\rightarrow \infty}\frac{h_\ell}\ell = 
  \frac{ \sigma_0\tan \phi + \gamma^2 r\left(\omega_\nu   +\omega_\tau\tan\phi \right)}{\tan\phi +2\gamma^2 },$$
where
   $\omega_\nu$ and $\omega_\tau$ are the $\nu$ and $\tau$ components of $\omega$, and $h_\ell$ is the height (ie. signed vertical displacement) after $\ell$ collisions. 
  The condition on the orbit for having transversal period $2$ does not restrict the values $\sigma_0, \omega_\nu, \omega_\tau$.
  Thus the motion is generically unbounded in the vertical direction. If the vertical drift is $0$, the motion is bounded. 
  \end{corollary}

We wish next to compare the no-slip billiard system in circular cylinders with the rolling process.  Let us first review the classical fact about bounded orbits for
the rolling motion.

\begin{proposition}\label{bounded rolling}
Suppose that  the cross-section of the $3$-dimensional vertical cylinder is a differentiable simple closed curve and that $\omega_e$\----the vertical component of the angular velocity vector, a constant of motion\----is non-zero.
Then trajectories of the rolling motion  under a constant force parallel to the axis of the cylinder are bounded.
\end{proposition}

Implicit in this statement is the assumption that the particle is constrained to remain in contact with the surface.

 \begin{wrapfigure}{R}{0.4\textwidth}
\centering
\includegraphics[width=2.0in]{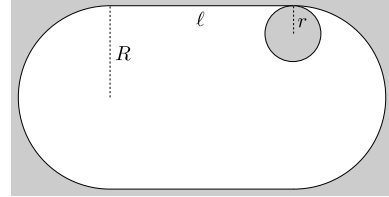}\ \ 
\caption{\label{stadium}{\small  Cross-section of the stadium cylinder  and  the rolling sphere. }}
\end{wrapfigure}

 As an illustration, consider the {\em stadium} cylinder whose cross-section, depicted in 
  Figure \ref{stadium}, consists of  two circular caps connected by parallel line segments.
 A glance at the second order equation for the height function $h$
  in Theorem \ref{equations 3 roll}  shows that the   ball is essentially in free fall (with acceleration $g/(1+\gamma^2)$)  while it rolls on the flat parts of the surface. In order to remain bounded, it  must rebound upward when it passes over the curved caps.
   Figure \ref{stadium_curve} shows a typical height function.  
   In the final section of the paper we will 
  briefly explore numerically the no-slip billiard version of this   example. It will be apparent that the question whether (and under which initial conditions) orbits remain bounded  is much more challenging when the transverse no-slip billiard system is chaotic.

  Let us return to the circular cylinder.
 Compared to  rolling motion, the behavior  of a no-slip billiard system inside a cylindrical billiard domain, in the presence of a constant force pulling the particle downward, seems to be  more subtle. On the one hand, it is possible, and typical for general cylinders, for the particle to accelerate downward (as one might   expect). See Figure \ref{falling graph}.

   \begin{figure}[htbp]
\begin{center}
\includegraphics[width=3.5in]{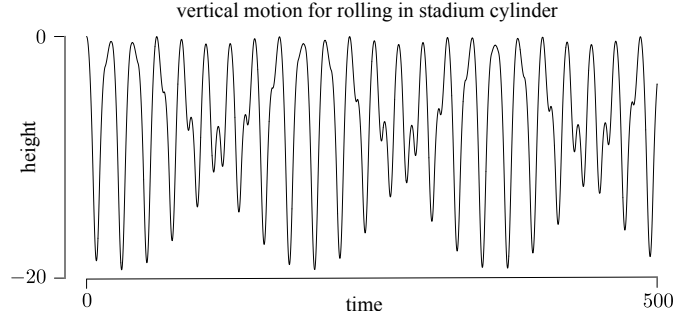}\ \ 
\caption{{\small  Height of center of mass of rolling particle in a cylinder with stadium cross-section. }}
\label{stadium_curve}
\end{center}
\end{figure}

 But   in dimension $3$ and for ordinary circular cylinders, we show that for a class of initial conditions satisfying what  we call {\em transversal rolling impact}, the particle does not fall: its position along the axis of the cylinder remains bounded.

 Prior to stating the definition of rolling impact, observe that if $(a,u,U)\in \mathcal{N}$ is the pre-collision state at a boundary point $a$,
 then the velocity of the point of contact $x= a-r\nu_a$ of the spherical billiard particle of radius $r$ and the boundary of $\mathcal{B}_0$
 is $v=u+U(x-a)=u-rU\nu_a$. We say that the collision satisfies the rolling impact condition if the component of $v$ tangent to the boundary of $\mathcal{B}$ at $a$ is zero. 
 \begin{definition}[Transversal rolling impact]\label{transversal rolling def}
 The pre-collision state $(a,u,U)\in \mathcal{N}$ will be said to satisfy the {\em rolling impact} condition if the orthogonal projection of 
 $v=u - rU\nu_a$ to $T_a\partial \mathcal{B}$  is zero. 
If the billiard domain is a cylinder (not necessarily circular) whose axis is parallel to the unit vector $e\in \mathbb{R}^n$,
  we say that  $(a, u, U)$ satisfies the {\em transversal rolling impact} condition if the orthogonal projection of
  $v$ to $e^\perp \cap T_a\partial \mathcal{B}$ is zero.
 \end{definition}

  \begin{figure}[htbp]
\begin{center}
\includegraphics[width=3.0in]{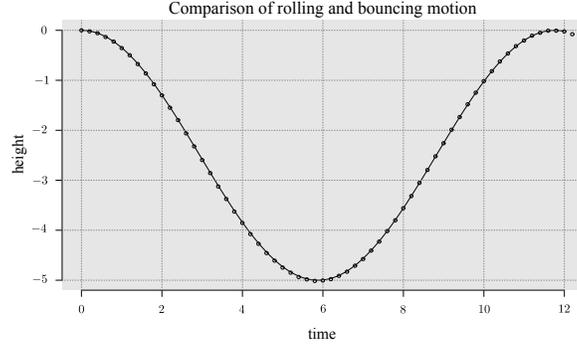}\ \ 
\caption{{\small  Comparison of the height functions for the rolling motion in a circular cylinder under a constant downward force (solid line) and the corresponding  no-slip billiard motion satisfying the transverse rolling impact condition (small circles). Initial conditions are chosen so that the two processes rotate around the cylinder at the same rate.}}
\label{roll and bounce}
\end{center}
\end{figure}

    The above definition of rolling impact is equivalent to the following: at each boundary configuration $q\in \partial M$ of the no-slip billiard,
  an initial state $v\in T_qM$ projects 
to a vector in the no-slip subspace $\mathfrak{S}_q$. Here we are referring to the orthogonal projection relative to the kinetic energy inner product on 
$T_qM$, whereas in Definition \ref{transversal rolling def} it is the ordinary Euclidean inner product that is being invoked. 
   Also notice that transversal rolling impact means that the rolling impact condition holds for the transversal billiard system, which is well-defined due to Theorem \ref{first theorem}.

 \begin{figure}[htbp]
\begin{center}
\includegraphics[width=4.0 in]{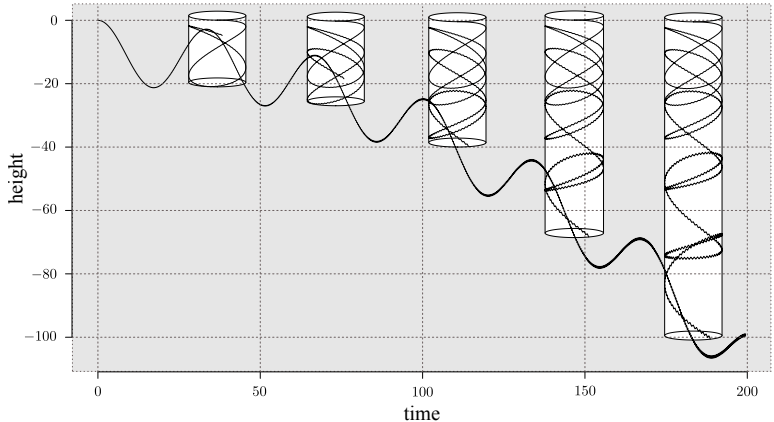}\ \ 
\caption{\small{When the transversal rolling impact condition does not hold (see Definition \ref{transversal rolling def}), the particle acquires an overall acceleration  downward. The apparent increase in thickness of the height function graph and of the particle's path is due to a small scale vertical zig-zag
motion of increasing amplitude. See also Figure \ref{cylinders comparison}.}} 
\label{falling graph}
\end{center}
\end{figure}

   For cylinders in dimension $3$, let $\tau_a$   denote the unit tangent vector to $\partial \mathcal{B}$ at $a$, oriented so that
   $\tau_a, \nu_a, e$ form a positive basis. Let   $\omega\in \mathbb{R}^3$  be the angular velocity vector.
   (We recall that $\omega$ is defined by  $\omega\times x:= Ux$ for all $x\in \mathbb{R}^3$.)
   Then the rolling impact condition is in this case expressed by the equation 
   $u-u\cdot \nu_a\nu_a = r\omega\times \nu_a $ and the transversal rolling impact condition by
\begin{equation}\label{transversal rolling impact}
u\cdot \tau_a +r\omega\cdot e=0
\end{equation}
as a simple algebraic manipulation involving the cross-product shows. We take as a measure of the failure  in satisfying this condition 
the quantity $-r\omega\cdot e/u\cdot \tau_a$ and call it the {\em transversal rolling defect} (evaluated on the initial velocities).

The following simple observation is essential.
 \begin{proposition}\label{equal times}
  Consider a two-dimensional no-slip billiard system in  a disc. If the first collision satisfies the rolling impact condition, then all subsequent collisions also do, and the times between consecutive  collisions are all equal. Furthermore, the center of mass of the moving particle undergoes specular reflection at each collision.
  \end{proposition}

  The main theorem for no-slip billiards in circular cylinders is now the following.

  \begin{theorem}\label{main}
Consider a no-slip billiard system in a circular cylinder in $\mathbb{R}^3$ whose moving particle is subject to a constant force directed along the axis of the cylinder. If the first collision satisfies the transversal rolling impact condition and the first flight segment does not go through the  axis of the cylinder, then the particle's trajectory is bounded. 
 \end{theorem}

 Figure \ref{roll and bounce} shows that the rolling motion and the billiard motion consisting of a sequence of short no-slip bounces, under the assumption  that the initial velocities satisfy the transversal rolling impact condition are, in fact, very close.

One may be inclined to think that the equations for the billiard motion are a simple-minded discretization of the differential equation for the rolling motion,  or that the latter
is a straightforward  limit of the former, but this seems  not to be the case. 
It is essential here to bear in mind the phenomenon illustrated in Figure \ref{falling}, which shows that if the transversal rolling defect is not $1$, then in the continuous limit one obtains a kind of motion that appears to be smooth but is   very different from that of solutions of the rolling  differential equation.  The small scale zig-zag motion shown more clearly in Figure \ref{cylinders comparison}, which has  a close-up of segments of trajectories highlighting  the effect of introducing a small transversal rolling defect, suggests a potential difficulty to overcome. It would be most interesting to find a limit differential equation containing this rolling defect as an equation  parameter, and the rolling equations as a special case when the parameter is $1$. Such an equation would, hopefully, suggest   a possible physical interpretation of this key parameter.  We leave this as a problem to be addressed in a future work.

 \begin{figure}[htbp]
\begin{center}
\includegraphics[width=3.5 in]{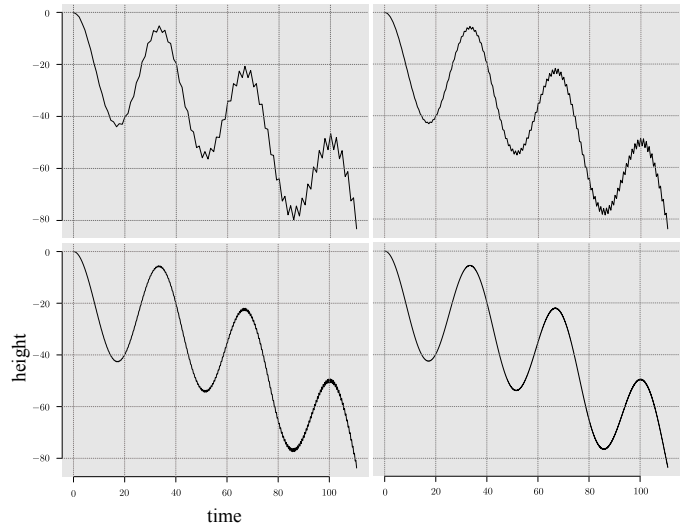}\ \ 
\caption{\small{These graphs correspond to a fixed initial transversal rolling defect $-r\omega\cdot e/u\cdot\tau_a=1.15.$ (When the  transversal rolling impact initial condition holds, this value is $1$.) 
Here $\omega\cdot e$ is the longitudinal component of the initial angular velocity vector and $u\cdot \tau_a$ is the tangential (to the boundary of the billiard domain) component of the center of mass velocity.
As the intercollision flight becomes shorter and motion grazes the cylinder more and more closely,
the height function becomes smooth but the falling rate remains essentially unchanged.}} 
\label{falling}
\end{center}
\end{figure}

The proofs of the above theorems and propositions, and of some of the more general facts to be stated shortly that are not restricted to dimension $3$,  will be given in the subsequent sections. Although our 
more complete observations pertain to $3$ dimensional domains, we have chosen to state and prove our results, whenever we can,  in arbitrary dimensions. We believe that this subject touches on a number of questions of geometric/dynamical interest,  for example, concerning  a possible theory of discrete non-holonomic systems, for which it would be too restrictive to remain in dimension $3$. Partly for this reason, but also  for the sake of completeness, we have also included reasonably detailed proofs of properties of rolling motion, such as the fact that   general cylinders in dimension $3$ have bounded orbits, even though this is a classical fact. We believe that our more geometric approach, which avoids relying too much on background material from mechanics  (like the use  of such concepts as Coriolis torque, for example; see \cite{gualtieri}) and highlights the central role   of the Euclidean group,  is worthwhile recording.

\begin{figure}[htbp]
\begin{center}
\includegraphics[width=3 in]{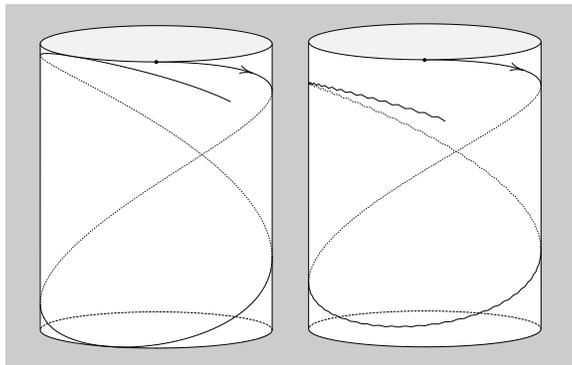}\ \ 
\caption{\small{Near grazing paths of no-slip billiard particle. On the left-hand side, the transversal rolling impact condition holds for the initial bounce, but on the right-hand side a small deviation of this condition is introduced. Notice the characteristic zig-zag nature of the curve and that it does not return all the way to the initial height. The starting center of mass position is indicated by the small black dot.}} 
\label{cylinders comparison}
\end{center}
\end{figure}

\section{More definitions and basic facts} \label{sec:more definitions and basic facts}
 There are two Riemannian metrics we need to consider: the one  on $M$ defined above by (\ref{metric}), and the metric on the hypersurface $S=\partial \mathcal{B}$ induced (by restriction) from the Euclidean metric in $\mathbb{R}^n$. 
 Although the context should make it clear which  is being referred to at any given moment, for example, when stating that vectors have unit length or are mutually orthogonal, we will always refer to the former as the  {\em kinetic energy metric}.

  We define the {\em cross-product} in $\mathbb{R}^n$ as the bilinear map $(a,b)\in \mathbb{R}^n\times \mathbb{R}^n\mapsto a\wedge b\in \mathfrak{so}(n)$ taking on values in the Lie algebra $\mathfrak{so}(n)$ of the rotation group, given by 
  $$u\mapsto (a\wedge b)u:= (a\cdot u)b-(b\cdot u)a$$   for all  $u\in \mathbb{R}^n.$ If $a$ and $b$ are unit orthogonal vectors and $V$ is the plane linearly spanned by $a$ and $b$, then
   $$\exp(\theta(a\wedge b))=\Pi^{\perp}+ (\cos \theta)\Pi +(\sin \theta) a\wedge b\in SO(n)$$ is a rotation matrix that restricts to the identity transformation on $V^\perp$ and rotates vectors in $V$ by angle $\theta$, where $\Pi$ and $\Pi^\perp$ are the orthogonal projections to $V$ and $V^\perp$, respectively.  (We also use  $\Pi$, possibly with additional sub- or superscripts, for    other orthogonal projections
   to appear  in the course of this paper. It will be clear in each case to which projection we are referring.) It is useful to note that
    \begin{equation}\label{wedge1}\frac12 \text{Tr}\left((a\wedge b)(c\wedge d)^\dagger\right) = (a\cdot c) (b\cdot d) - (a\cdot d)( b\cdot c) \end{equation}
   where $\dagger$ denotes transpose.
   If  moreover $U$  is an $n\times n$ matrix then
\begin{equation}\label{wedge2}\frac12 \text{Tr}\left(U(a\wedge b)^\dagger\right) = (Ua)\cdot b. \end{equation} 
   Notice that, if $n=3$, $$(a\wedge b) u=(a\times b)\times u,$$ where $\times$ is the ordinary cross-product of vector algebra.

  The kinetic energy Riemannian metric (\ref{metric}) is a product metric on $M=SO(n)\times \mathcal{B}$. The trace part
 is a bi-invariant  metric on the rotation group.   It is a standard and easy to prove fact that the Levi-Civita connection on $SO(n)$ associated to this metric satisfies
 $$\nabla_XY=\frac12[X,Y]$$
  for any pair $X,Y$ of left-invariant vector fields.  The following proposition is also a standard fact about bi-invariant metrics.

  \begin{proposition}\label{derivative}
  Let $A(t)$ be a smooth path in $SO(n)$ where the rotation group is given the above bi-invariant trace metric.  Define $U(t)=A(t)^{-1}\dot{A}(t)\in \mathfrak{so}(n)$, where
  $\dot{A}$ indicates derivative in $t$ of the matrix-valued function. Then 
  $$ \frac{\nabla \dot{A}}{dt}= A\dot{U}.$$
  In particular, $A(t)$ is a geodesic iff $U(t)$ is constant and $A(t)=A(0)e^{tU}.$
  \end{proposition}

    We often  find it convenient   to express tangent vectors  to the Euclidean group $SE(n)$
    at a point $(A,a)$
     in the form
    $(UA, u)$, where $U\in \frak{so}(n)$ and $u\in \mathbb{R}^n$.     In this form, the vector is obtained by a right-translation from the identity $(I,0)$ to $(A,a)$ of an element of the
    Lie algebra $\mathfrak{se}(n)$ of $SE(n)$.  
   The reader should be attentive to the distinction between $(UA,u)$ and $(AU,Au)$. The latter arises when we wish to think of $(U,u)\in \mathfrak{so}(n)$ as an infinitesimal rotation expressed in the so-called {\em body frame}; on the other hand, when writing a tangent  vector as $(UA, u)$,  one has the following interpretation:
  a material point $b\in D$ which in configuration $(A,a)$ is at $x=Ab+a\in \mathbb{R}^n$, will have velocity at $x$ equal to
  $U(x-a)+u.$ In particular, $a$ and $u$ are, respectively,  the position  and the velocity of the center  of mass of the ball $D_0(r)$ in the {\em state} $(UA,u)\in T_{(A,a)}M$.

    If $A(t)$ is a smooth curve in $SO(n)$ and $U(t)$ is now defined by $U(t)= \dot{A}(t)A(t)^{-1}$, then
 by Proposition \ref{derivative}   $$A^{-1}\frac{\nabla \dot{A}}{dt}= \frac{d}{dt}\left(A^{-1}\dot{A}\right)= \frac{d}{dt}\left(A^{-1}U A\right)= -\left(-A^{-1}\dot{A}A^{-1}\right)UA + A^{-1}\dot{U}A+ A^{-1}U\dot{A},$$
    and the last term simplifies to $-A^{-1}U^2A+A^{-1}\dot{U}A+ A^{-1}U^2A=A^{-1}\dot{U}A.$ Thus we immediately have the following remark:
    \begin{proposition}\label{derivatives}
    For a smooth curve $q(t)=(A(t),a(t))$   in $M$, write $\dot{q}=(U(t)A(t), u(t))$, where $u=\dot{a}$ and $U=\dot{A}A^{-1}$.  Then the acceleration
    of $q(t)$ relative to the Levi-Civita connection $\nabla$ of the left-invariant (kinetic energy) metric (\ref{metric}) is 
    $$ \frac{\nabla \dot{q}}{dt}=\left(\dot{U}A, \dot{u}\right).$$
    When the focus is on the geometry of the boundary of $M$ with the induced metric, and $q(t)\in S:= \partial M$, then $\dot{u}$ is replaced with $\frac{D u}{dt}=\dot{u}-(\dot{u}\cdot \nu_a) \nu_a $ where $D$ is the Levi-Civita connection of the hypersurface $S\subset \mathbb{R}^n$ and $\nu_a$ is a unit normal vector to $S$ at $a$. 
    \end{proposition}

    Recall from Definition \ref{no slip bundle} the no-slip bundle $\mathfrak{S}$.
     A simple computation gives the explicit form of its orthogonal complement  relative to the kinetic energy metric. Notice that the unit normal vector to
    $\partial M$ pointing into $M$ is $\mathbbm{n}(A,a)= \frac1{\sqrt{m}} (0, \nu_a)$, where $m$ is the mass of the moving particle. Clearly, $\mathbbm{n}(q)$ lies  in that orthogonal complement. However, in what follows, it will be convenient  to reserve the notation $\mathfrak{S}^\perp_q$ for the orthogonal vectors to $\mathfrak{S}_q$ contained in $T_q(\partial M)$.  Thus defined, we have
    \begin{equation}\label{Sperp}\mathfrak{S}^\perp_{(A,a)}=\left\{\left(\frac1{r\gamma^2} (w\wedge \nu_a) A,w\right): w\in T_aS\right\}.
    \end{equation}
    Properties (\ref{wedge1}) and (\ref{wedge2}) of the product $\wedge$ are useful for verifications of this kind.

  We give now a more explicit description of the no-slip collision map on the reduced phase space $\mathcal{N}$ introduced in Definition \ref{Cq defined}. For details we refer the reader to 
  \cite{CF}. The abbreviations $c_\beta$ and $s_\beta$ will be used throughout the rest of the paper for the   quantities:
 $$c_\beta:=\cos\beta:=\frac{1-\gamma^2}{1+\gamma^2}, \ \ s_\beta:=\sin\beta := \frac{2\gamma}{1+\gamma^2}.$$
The angle $\beta$  is defined by these relations. When there is no chance of confusion we may simply write $c, s$. 
Recall that $\nu_a$ is the unit inward pointing normal vector to $S$.

  \begin{proposition}[No-slip collision map, \cite{CF}]\label{no-slip map}   For each $a\in S$, the no-slip collision map at $a$ is the linear map $C_a:\mathcal{N}_a\rightarrow \mathcal{N}_a$  such that
$$C_a(u,U)= \left(c_\beta u -\frac{s_\beta}{\gamma} (u\cdot \nu_a) \nu_a + s_\beta\gamma r U\nu_a    ,\frac{s_\beta}{\gamma r}\nu_a\wedge u + U -  \frac{s_\beta}{\gamma}\nu_a \wedge U \nu_a\right). $$
  \end{proposition}

 Finally, let us also recall the following definition. 
  \begin{definition}[Shape operator]\label{shape}
 The {\em shape operator} $\mathbb{S}_a$ of the hypersurface $S\subset \mathbb{R}^n$ at $a$ is   the symmetric  linear transformation of  $T_aS$ mapping  $v$ to  $-D_v\nu$, where $\nu$ is the vector field along $S$ of normal vectors and
  $D_v$ is  being used here for     the  ordinary directional differentiation of $\mathbb{R}^n$-valued functions.  (Elsewhere in the paper we also use $D$  for the  Levi-Civita connection on $S$ relative to t he metric induced by the ordinary dot product.) The unit eigenvectors of $\mathbb{S}_a$ are called the {\em principal vectors} at $a$, and the eigenvalues are the {\em principal curvatures.}
  \end{definition}

  \section{Rolling motion on   cylinders} \label{sec:rolling motion on cylinders}
In this  section we review general facts about rolling and give a self-contained discussion of the fact that orbits of the rolling motion 
on $3$-dimensional cylinders with general cross-section. This is a classical result in non-holonomic dynamics, but we wish briefly to 
re-derive it here (in the present more geometric setting) so as to   see more clearly the parallels with  similar bounded motion of the   no-slip billiard counterpart.  Throughout the section $e$ will denote the axis  vector of the cylinder, as in Definition \ref{cylinder}.  Moreover, $\nu(a) = \nu_a$ continues to denote the inward pointing unit normal vector.    In order to simplify notation, we will write $\nu$ except when emphasizing the dependence on boundary point.

The rolling particle is  subject to a force $f$ (a vector field on $M$), assumed to have zero $\mathfrak{so}(n)$ component
and  $\mathbb{R}^n$ component $\varphi e$, where $\varphi=-m g$ is constant and $g$ is interpreted  as the acceleration due to gravity. Such an  $f$ does not directly affect the angular velocities and can be thought to act   on the center of mass of the moving particle.

  Returning to the constrained Newton's equation of Definition \ref{Newton}, notice that due to the description of $\mathfrak{S}^\perp$ in  Equation (\ref{Sperp}), 
  the constraint force $N\in \mathfrak{S}^\perp$ at $q=(A,a)$ has the form $$N=\left(\frac1{r\gamma^2}\zeta \wedge \nu A, \zeta\right) $$
  where $\zeta=\zeta(q,\dot{q})\in T_aS$ must  be  determined from  the  constraint $\dot{q}\in \mathfrak{S}_q$. 
  Using Proposition \ref{derivatives},  and writing $\dot{q}=(UA, u)$, where $U\in \mathfrak{so}(n)$ and $u\in T_aS$, we can split Newton's equation into separate equations for the angular velocity matrix $U$ and center of mass velocity $u$:
  \begin{equation}\label{Newtonsplit}
  \dot{U}=\frac{1}{r\gamma^2} \zeta \wedge \nu, \ \ \frac{Du}{dt}= -g e +\zeta.
  \end{equation}
  As before, $Du/dt$ refers to the Levi-Civita connection on $S$ relative to the metric induced by the standard Euclidean metric (dot product) on $\mathbb{R}^n$.  The first equation and the definition of $\wedge$ imply
  \begin{equation}\label{zeta}\zeta = -r\gamma^2 \dot{U} \nu\end{equation}
 while  the constraint on $\dot{q}$  (recall the expression for $\mathfrak{S}$ given in (\ref{S})) implies
  \begin{equation}\label{uconstraint}
  u = rU\nu.
  \end{equation}

  It is not difficult to solve for $\zeta$ and  obtain an  explicit differential equation for $U(t)$. This is done in the next proposition.

  \begin{proposition} \label{deNewton}The path  $q(t)=(A(t), a(t))$ that solves the  constrained Newton's equation for the rolling process  is the solution to
   an initial value problem for the system of the differential equations
$$
  \dot{a}=r U \nu_a, \ \ 
  \dot{A}=UA, \ \ 
  \dot{U}=-\frac1{(1+\gamma^2)r} \left( {r^2}U\mathbb{S}_a U\nu_a -{g}e \right)\wedge \nu_a. 
$$
  \end{proposition}
  \begin{proof} 
  Let us define for each $a\in S$ the linear map
  \begin{equation}\label{Gamma}\Gamma_a: U\in \mathfrak{so}(n)\mapsto U-{\gamma^{-2}}\left(U\nu_a\right)\wedge\nu_a\in \mathfrak{so}(n).\end{equation}
  This is an invertible map. To check this claim,  let $\tau_1, \dots, \tau_n$ be an orthonormal basis of $\mathbb{R}^n$ with $\tau_n=\nu_a$. 
  Then a straightforward computation using the basic properties of $\wedge$ leads, for $i<j\leq n$,  to
  \begin{equation}\label{Gamma entries}
  \tau_i\cdot \Gamma_a(U)\tau_j =\left(1+\frac{\delta_{jn}}{\gamma^2}\right)\tau_i\cdot U\tau_j.
  \end{equation}
  Clearly, then, $\Gamma_a$ is injective, hence invertible. It may be helpful  noting here that
  $\tau_i\cdot U\tau_j$ is a constant multiple of the trace inner product $\left\langle U, \tau_i\wedge \tau_j\right\rangle$.
  If in particular  $U=w\wedge \nu$,  then $\Gamma(w\wedge \nu)= (1+\gamma^2)/\gamma^2 w\wedge \nu.$

  The first equation, for $\dot{a}$, comes from (\ref{uconstraint}) and the second, for $\dot{A}$, is immediate from the
  definition $U=\dot{A}A^{-1}$. Thus we only need to justify the third equation.  The second equation in  (\ref{Newtonsplit}) gives
  $\zeta= \frac{Du}{dt}+ge$, and  the first   gives 
  $$\dot{U}=\frac1{r\gamma^2}\left(\frac{Du}{dt}+ge\right)\wedge \nu.$$
  In the above we may replace the covariant derivative $\frac{Du}{dt}$ with the ordinary derivative $\dot{u}$ since the difference is a multiple of
  $\nu$, which vanishes after taking the cross product $\wedge$  with $ \nu$. 
Now notice that
  $ \dot{\nu}_a = D_u\nu = - \mathbb{S}_au$, where $\mathbb{S}_a$ is the shape operator of $S$ at $a$. (See Definition \ref{shape}.)
  Therefore,
  $$ \dot{u}=r\dot{U}\nu + rU\dot{\nu} = r\dot{U}\nu - rU\mathbb{S}u =  r\dot{U}\nu - r^2U\mathbb{S}U\nu,$$
  from which we obtain
  $$\Gamma(\dot{U})= \dot{U}-\frac1{\gamma^2}\left(\dot{U}\nu\right)\wedge\nu = -\frac1{r\gamma^2}\left(r^2U\mathbb{S}U\nu -ge\right)\wedge\nu.$$
  The desired third equation follows from the definition of  $\Gamma$. 
  \end{proof}

  Since we are particularly concerned with the issue of boundedness of trajectories, we need to obtain information about the function $h(a)=a\cdot e$, whose derivative in time is $u\cdot e$. The next proposition shows how to get a handle on this quantity.
  \begin{proposition}\label{main equation n}
Let  $e$ be the axis vector of the cylinder $S$, $\nu$  the inward pointing unit normal vector field  of  $S$, $q(t)=(A(t), a(t))$, and $\dot{q}(t)=(U(t)A(t), u(t))$. Then
  $$\frac{d}{dt}(e\cdot u)= \frac{\gamma^2}{1+\gamma^2}r^2 \left(Ue\right)\cdot \left(\mathbb{S}_a U\nu\right) -\frac{g}{1+\gamma^2}$$
  holds under the assumption that $q(t)$ satisfies the constrained Newton's equation for the rolling motion of the particle of radius $r$, mass $m$, mass distribution parameter $\gamma$, subject to a constant force $-mge$. 
  Furthermore,
   $$\left(Ue\right)\cdot \left(\mathbb{S}_a U\nu\right)=\sum^{n-2}_{i=1}\lambda_i(a)(\tau_i\cdot Ue)(\tau_i\cdot U\nu) $$
  where $\tau_i=\tau_i(a)$, $i=1,\dots, n-1$, is an orthonormal basis of $T_aS$ consisting of principal vectors, $\lambda_i(a)$ are the respective principal curvatures of $S$, and
$\tau_{n-1}=e$, $\lambda_{n-1}=0$. In dimension $3$ we have $\tau_a:=\tau_1(a)=\nu_a\times e$. Letting in this case $\lambda(a)$ be the principal curvature in direction $\tau_a$, the above equation reduces to
$$\frac{d}{dt}(e\cdot u)= \frac{\gamma^2}{1+\gamma^2}r^2 \lambda(a)\left(\tau_a\cdot Ue\right)  \left(\tau_a\cdot U\nu_a\right) -\frac{g}{1+\gamma^2}. $$
  \end{proposition}
  \begin{proof}
  By Proposition \ref{deNewton}  
  $$r e\cdot \dot{U}\nu =\frac{1}{1+\gamma^2}\left[-r^2\left(Ue\right)\cdot\left(\mathbb{S}U\nu\right)-g\right]. $$
   Now observe that
  $$ \frac{d (e\cdot u)}{dt}  = r\frac{d(U\nu)}{dt}= re\cdot \dot{U}\nu + r e\cdot U \dot{\nu}= re\cdot \dot{U}\nu - r e\cdot U \mathbb{S} u
  = re\cdot \dot{U}\nu - r^2 e\cdot U \mathbb{S} U\nu$$
  hence $re\cdot \dot{U} \nu = \frac{d (e\cdot u)}{dt} -r^2 (Ue)\cdot (\mathbb{S} U\nu)$.  The first equation of the proposition is now an immediate consequence.
  The other claims follow easily from definitions.  \end{proof}
  
In dimension $3$  there is an orthonormal moving frame consisting of $\tau, \nu, e$ where $\tau=\nu\times e$. The field $\tau$
  is  parallel  ($\frac{D\tau}{dt}=0$). Furthermore $\mathbb{S}_a\tau_a=\lambda(a)\tau_a$ where $\lambda(a)$ is a principal curvature, and $\mathbb{S}_ae=0$. When $S$ is a circular cylinder of radius $R-r$ (and the extended billiard domain is a circular solid cylinder of radius $R$), $\lambda(a)=1/(R-r)$.
  
  \begin{theorem} \label{equations 3 roll} Let $n=3$ and introduce the {\em angular velocity vector} $\omega$ according to the definition
$w\mapsto U=\omega\times u.$ 
Under the conditions and notations of Proposition \ref{main equation n},  the following system of equations hold: 
  $$ \frac{d}{dt}\left(u\cdot e\right) +\frac{\gamma^2}{1+\gamma^2} r^2 \lambda(a) (\omega\cdot e)(\omega\cdot \nu) +\frac{g}{1+\gamma^2}=0$$
and 
  $$ \frac{d}{dt}\left( \omega\cdot \nu\right) = \lambda(a)(\omega\cdot e)(u\cdot e),$$
  where $\omega\cdot e$ is constant in $t$.  Notice that $u\cdot e=\dot{h}$, where  $h(a)=a\cdot e$  is the height  of the center of mass of the moving particle.  
  In terms of $h$,  the above becomes
  $$ \ddot{h}+\frac{\gamma^2}{1+\gamma^2} r^2 \lambda(a)\omega_e(0)\left[\omega_\nu(0)+\omega_e(0)\int_0^t\lambda(a(s))\dot{h}(s)\, ds\right]+\frac{g}{1+\gamma^2}=0,$$
  where $\omega_e=\omega\cdot e$ and  $\omega_\nu=\omega\cdot \nu$.
  \end{theorem}
  \begin{proof}
 Observe that $$\nu\cdot \mathbb{S} U\nu= 0, \  e\cdot \mathbb{S}U\nu=(\mathbb{S} e)\cdot (U\nu)=0,  \ \tau\cdot \mathbb{S}U\nu=(\mathbb{S}\tau)\cdot (U\nu)=\lambda \tau\cdot   U\nu,$$ so that
 $$(Ue)\cdot(\mathbb{S}_a U\nu)= (\tau\cdot Ue)(\tau\cdot\mathbb{S}U\nu)+ (\nu\cdot Ue)(\nu\cdot \mathbb{S}U\nu)+(e\cdot Ue)(e\cdot \mathbb{S}U\nu)=
 \lambda (\tau\cdot Ue)(\tau\cdot U\nu).$$
 The quantity $\tau\cdot U\nu$ is constant in $t$. To verify this claim, first observe that, as $\dot{\tau}$ is collinear  to $\nu$ and $\dot{\nu}$ is collinear to $\tau$,
 $$\frac{d}{dt}\left(\tau\cdot  U\nu\right)= \dot{\tau}\cdot U\nu+\tau\cdot \dot{U}\nu+\tau\cdot U \dot{\nu} =\tau\cdot \dot{U}\nu  $$
From the third equation in Proposition \ref{deNewton} and Equation (\ref{Gamma entries})
we obtain
$$\left(1+\frac1{\gamma^2}\right)\tau\cdot \dot{U}\nu= \tau\cdot \Gamma(\dot{U})\nu=-\frac1{r\gamma^2}\tau\cdot\left[\left(r^2U\mathbb{S}U\nu-ge\right)
\wedge\nu\right]\nu=\frac{r}{\gamma^2}\tau\cdot U\mathbb{S}U\nu.$$
But $\tau\cdot U\mathbb{S}U\nu=-(U\tau)\cdot (\mathbb{S} U\nu)=-\lambda(\tau\cdot U\tau)(\tau\cdot U\nu)=0.$ Therefore $\tau\cdot U\nu$ is indeed constant.
It remains to understand the term $\tau\cdot Ue$ in
 $(Ue)\cdot(\mathbb{S}_a U\nu)=
 \lambda (\tau\cdot Ue)(\tau\cdot U\nu).$
 Observe that $\tau\cdot \dot{U} e=\tau\cdot \left((r\gamma^2)^{-1}\zeta\wedge\nu\right)e=0$ and $\dot{\tau}=\lambda u\cdot \tau \nu$, hence
 $$\frac{d}{dt}\left(\tau\cdot U e\right)= \dot{\tau}\cdot U e + \tau\cdot \dot{U}e=\lambda (\tau\cdot u)( \nu\cdot Ue)=-\lambda r (\tau\cdot U\nu)(e\cdot U\nu)=
 -\lambda(\tau\cdot U \nu)(e\cdot u),$$
where the third  and fourth equalities made use of Equation (\ref{uconstraint}). 
Finally, notice that  $\tau\cdot U\nu=-\omega\cdot e$ and $\tau \cdot U e=\omega\cdot \nu$. 
We conclude the proof by applying these observations to the last equation in the statement of   Proposition \ref{main equation n}.
 \end{proof}

As a simple example, we see that 
the height  
 of the rolling particle in a $3$-dimensional vertical circular cylinder undergoes simple harmonic oscillations, so  long as
 the constant of motion $\omega_e$ is non-zero. In particular, the motion  is bounded.
In fact,
  suppose the cross-section of  $S$ is a circle  of radius $R-r$. In this case $\lambda=1/(R-r)$ is constant, so
  $$\ddot{h} +\frac{\gamma^2}{1+\gamma^2}\left(\frac{r}{R-r}\right)^2 \omega_e(0)^2 h + \frac{\gamma^2}{1+\gamma^2}\frac{r^2}{R-r}\omega_e(0)
  \left(\omega_\nu(0)-\frac1{R-r}\omega_e(0) h(0)\right)+\frac{g}{1+\gamma^2}=0.$$
  This has  the form $\ddot{h}+c_0h +c_1=0$ where $c_0$ is a positive constant (assuming $\omega_e(0)\neq 0$).  In terms of the variable $z:=h+c_1/c_0$,
  the equation takes the form $\ddot{z}+c_0z=0$, whose solutions are the bounded functions $z(t)=C_1\cos(\sqrt{c_0}t)+C_2\sin(\sqrt{c_0}t).$

  The following interesting observation was made in \cite{gualtieri}. Let  $T_h$ and $T_v$ denote, respectively, the periods of horizontal and vertical oscillation of the rolling ball in the circular vertical cylinder. One easily finds that
  $T_h = 2\pi (R-r)/r\omega_e$ and $T_v=\sqrt{\frac{1+\gamma^2}{\gamma^2}}2\pi (R-r)/r\omega_e$. Therefore the ratio  of these two periods only depends on the mass distribution parameter $\gamma$: $T_v/T_h= \sqrt{\frac{1+\gamma^2}{\gamma^2}}$. For example,
   $\gamma^2=2/5$ for the uniform distribution in dimension $3$, so the period ratio in this case is $\sqrt{7/2}$.

We now restate and prove Proposition \ref{bounded rolling}. 
\newtheorem*{prop:bounded rolling}{Proposition \ref{bounded rolling}}
\begin{prop:bounded rolling}
Suppose that  the cross-section of the $3$-dimensional vertical cylinder is a differentiable simple closed curve and that the constant of motion $\omega_e$\----the vertical component of the angular velocity vector\----is non-zero.
Then trajectories of the rolling motion  under a constant force parallel to the axis of the cylinder are bounded.
\end{prop:bounded rolling}
\begin{proof}
Let $h=a\cdot e$ denote, as before, the height of the center of mass of the rolling particle, and introduce $\sigma=\dot{h}$ and $w$ the $\nu$ component of the angular velocity vector  $\omega$. According to Theorem \ref{equations 3 roll}, the function $h(t)$ can be obtained by solving an initial value problem
for the system
$$\dot{h} = \sigma, \ \ \dot{\sigma}= -c_1\lambda(t) w + c_3, \ \ \dot{w}= c_2\lambda(t)\sigma, $$
where $c_1, c_2, c_3$ are constants involving the parameters $\gamma$, $\omega_e$, $r$ and $g$, and $c_1, c_2$ are positive. The principal curvature
 $\lambda(t)=\lambda(a(t))$ is a periodic function of $t$ which is  known in advance since it only depends on the point of contact  at time $t$ along the cross-sectional boundary curve, and we know which point that is from the initial condition and the constant value of $\omega_e$. (That boundary point moves at a constant rate $r\omega_e$.) A simple rescaling of the variables gives the system
 $$\dot{x}_1= x_2, \ \ \dot{x}_2= -\eta(t) x_3 + 1, \ \ \dot{x}_3=\eta(t) x_2$$
 where $x_1$ is a constant multiple of $h$ and $\eta(t)$ is a periodic function of $t$ whose period we may assume without loss of generality to be $1$.
Introducing the complex variable $z=x_2+ix_3$, we previous system reduces to
$\dot{x}_1= \text{Re}(z), \ \ \dot{z}= i\eta z+1.$  For simplicity, let us assume $0$ initial conditions.  Then the differential equation for $z$ has solution
$$z(t) = e^{i f(t)}\int_0^t e^{-i f(s)}\, ds,$$
where $f(t)=\int_0^t\eta(s)\, ds$ satisfies $f(t+1)=f(t) + 1$. Standard integral manipulations give
$$ z(n+t)=\frac{e^{in}-1}{1-e^{-i}} e^{i f(t)} \int_0^1 e^{-if(s)}\, ds + e^{i f(t)}\int_0^t e^{-i f(s)}\, ds$$
for all integers $n$ and $0\leq t<1$.  The goal is to establish that   the real part of $\int_0^t z(s)\, ds$, which equals $x_1(t)$,  is a bounded function. 
Let us verify this fact for $t=n$, an integer. Another straightforward manipulation of integrals leads to
\begin{equation}\label{braces}\int_0^n z(s)\, ds = \text{(bounded term)} -\frac{n}{\left| 1- e^{-i} \right|^2}\left\{(1-e^i)\mathcal{I}_1- 2(1-\cos 1)\mathcal{I}_2  \right\}. \end{equation}
where $\mathcal{I}_1=\int_0^1\int_0^1 e^{i[f(t)-f(s)]}\, ds\, dt $ and $\mathcal{I}_2=\int_0^1\int_0^t e^{i[f(t)-f(s)]}\, ds\, dt$.
But $$\mathcal{I}_1=2\int_0^1\int_0^t \cos(f(t)-f(s))\, ds\, dt. $$ One then notices that the real part of the term in braces in equation (\ref{braces})
must be zero.
\end{proof}

  \section{No-slip billiards in  general cylinders}\label{sec:No-slip billiards in  general cylinders}
  We now prove Theorem \ref{first theorem}, reproduced below.  
\newtheorem*{thm:first theorem}{Theorem \ref{first theorem}}
\begin{thm:first theorem}
Let  $\mathcal{N}$ be the reduced phase space of the no-slip billiard system on the solid cylinder domain $\mathcal{B}\subset \mathbb{R}^n$, and let
$\overline{\mathcal{N}}$ be the reduced phase space for the associated  transverse billiard system. Then trajectories of the no-slip billiard on $\mathcal{N}$, possibly with a constant force in the longitudinal direction,  project to trajectories of the no-slip billiard map on $\overline{\mathcal{N}}$, where the latter system is given the same mass distribution parameter $\gamma$ as the billiard in dimension $n$. 
 \end{thm:first theorem}
\begin{proof}
  Given a vector space $W$, it makes sense to write the Lie algebra of the special Euclidean group on $W$,  as a vector space, in the form
  $\mathfrak{se}(W)=(W\wedge W) \oplus W$
  where $W$ corresponds to infinitesimal translations and $W\wedge W$ is the space spanned by elements $u\wedge v$, for all $u, v\in W$. In this notation we have, for $W=\mathbb{R}^{n-1} = e^\perp$, 
  $$\mathfrak{se}(n)= \mathfrak{se}(n-1)\oplus \left(\mathbb{R}^{n-1}\wedge e\right) \oplus \mathbb{R} e$$
  and this direct sum decomposition is orthogonal with respect to 
 the inner product  on $\mathfrak{se}(n)$ given above in Equation \ref{metric}.
 Also observe that the map $C_a$ (Proposition \ref{no-slip map}), at each $a$ on the boundary of $\mathcal{B}$, respects the decomposition 
 $$\mathfrak{se}(n)=\mathfrak{se}(n-1)\oplus  \mathfrak{se}(n-1)^\perp$$
since   for  all $w\in W$ for which $w\cdot \nu_a=0$ we have 
 $ C_a(0, w\wedge e)= (0,w\wedge e)$ and 
\begin{equation}\label{coll map}
 C_a\left(0, {\nu_a\wedge e}\right)=\left({r\gamma}s_\beta  e, -c_\beta {\nu_a\wedge e}\right), \ \ 
 C_a(e,0) = \left(c_\beta e, ({r\gamma})^{-1}s_\beta {\nu_a\wedge e}\right).
\end{equation}
 Here we are writing $C_a(u,U)$, for $u\in \mathbb{R}^n$ and $U\in \mathfrak{so}(n)$, on the reduced phase space $\mathcal{N}$.
 Letting $C_{\bar{a}}$ be the no-slip  reflection map at $\bar{a}$ of the system on $\overline{\mathcal{B}}$,  and writing $\Pi$ for the orthogonal projection from $\mathfrak{se}(n)$ to $\mathfrak{se}(n-1)$, it follows that 
 $$\Pi\circ C_a= C_{\bar{a}}\circ \Pi. $$
 (As already noted, we use the symbol  $\Pi$   to denote the orthogonal projection on various  subspaces of $\mathbb{R}^n$; the context will make it clear which  subspace  one is referring to at any given moment.)
From these observations we conclude that the natural projection from $\mathcal{N}$ to the reduced phase space $\overline{\mathcal{N}}$ of the transverse billiard system  commutes with the respective no-slip billiard maps.
\end{proof}

 The following notation will be used in
our  study of the longitudinal motion   of no-slip billiards in general cylinders with axis vector $e$. 
  Let   $(a_j, u^-_j, U^-_j), (a_j, u_j, U_j)\in \mathcal{N}$ denote, respectively, the pre- and post-collision states  at the $j$th collision, $j=0, 1, \dots$,
  for a given trajectory of the no-slip system in  $\mathcal{B}$; 
   we  write  $\nu_j=\nu(a_j)$ for the inward pointing normal vector to the boundary of $\mathcal{B}$ at $a_j$; the time interval  between consecutive collisions, from the $j$th to the $j+1$st collision, will be denoted $t_j$; we further introduce the velocity components $\sigma_j:= u_j\cdot e$ and $w_j:= \gamma  rU_je\in e^\perp=:W$, and the longitudinal projection $h_j:=a_j\cdot e$.  
  Define the following elements of $\mathbb{R}^n=\mathbb{R}e\oplus W$:
  $$\Lambda_i=\left(\begin{array}{c}\sigma_i \\w_i\end{array}\right), \ \ \Lambda_i^-=\left(\begin{array}{c}\sigma^-_i \\w^-_i\end{array}\right), \ \ \mathbbm{1}=\left(\begin{array}{c}1 \\0\end{array}\right),\ \  \Phi=-g\mathbbm{1}.$$
  The
 special notation $\mathbbm{1}$ is used here  for the first standard basis vector of $\mathbb{R}^n$ in order to emphasize that we are dealing with a  velocity space mixing linear and angular components, and not the ambient $\mathbb{R}^n$  of the billiard domain. 
Set $W_a:=\{w\in W: w\cdot \nu_a=0\}$,  $a\in \partial\mathcal{B}$,  let $\Pi_a$ be the orthogonal projection to $W_a$, and define
 \begingroup
\renewcommand*{\arraystretch}{1.5}
$$\mathcal{A}(a):=\left(\begin{array}{cc}c_\beta & -s_\beta \nu_a^\dagger \\-s_\beta \nu_a & -c_\beta \nu_a \nu_a^\dagger + \Pi_a\end{array}\right) $$
\endgroup 
  Simple algebraic manipulation using  the basic properties of $\wedge$ and Proposition \ref{no-slip map} gives that  $\mathcal{A}(a)$ maps pre- to post-collision velocity components in the mixed velocity space $\mathbb{R}^n$. Thus 
  $\Lambda_i=\mathcal{A}_i\Lambda_i^-$, where $\mathcal{A}_i:=\mathcal{A}(a_i)$. Over the intercollision flight,  the change in these $n$ mixed velocity components
is: $\Lambda^-_i=\Lambda_{i-1}+t_{i-1}\Phi$ since $\sigma_i^-=\sigma_{i-1}- t_{i-1}g$ and $w_i^-=w_{i-1}$ (recall that $U$ does not change between collisions). Therefore,

  \begin{proposition}\label{longitudinal equations}
 With the notation just introduced, the sequence of displacements $h_i$ along the cylinder's axis 
 satisfies 
 {\em \begin{align*}h_i&=h_{i-1}+\mathbbm{1}^\dagger\left(t_{i-1}\Lambda_{i-1}+\frac{t^2_{i-1}}{2}\Phi\right) \\
 \Lambda_i&=\mathcal{A}_i\left(\Lambda_{i-1}+t_{i-1}\Phi\right).
 \end{align*}}
with  initial conditions $\Lambda_0$ and $h_0$. 
  \end{proposition}
 \begin{proof}
 This is a simple  consequence of the general form of $C_a$ given in  Proposition \ref{no-slip map}, the above definitions, and the elementary properties of $\wedge$. 
 \end{proof}

 Observe that $\mathcal{A}(a)$, like $C_a$, is an orthogonal involution. It has eigenvalues $-1$ with multiplicity $1$ and $1$ with multiplicity $n-1$. In fact we have
for all $w\in W_a$
 $$\mathcal{A}(a) \left(\begin{array}{c}0 \\ w\end{array}\right)= \left(\begin{array}{c}0 \\ w\end{array}\right),\ \  \mathcal{A}(a)  \left(\begin{array}{c}-1 \\  \gamma \nu_a\end{array}\right)= \left(\begin{array}{c}-1 \\  \gamma \nu_a\end{array}\right), \ \ 
 \mathcal{A}(a)  \left(\begin{array}{c}\gamma \\    \nu_a\end{array}\right)= - \left(\begin{array}{c}\gamma \\    \nu_a\end{array}\right).$$

 We turn our attention now to the longitudinal motion when the no-slip billiard orbit is transversely periodic of period $2$. We wish to find an expression for  the longitudinal drift in the absence of forces. This is provided by the following theorem.

  \begin{theorem}\label{drift} For  an orbit with transversal period $2$,   define $Q=\mathcal{A}_1\mathcal{A}_2$, $\mathcal{A}=\mathcal{A}_1$,
  and the row vector $\xi = \left(1+c_\beta, -s_\beta \nu_1^\dagger\right)$.   Then $Q\in SO(n)$. Denote by $P$ the orthogonal projection  onto the eigenspace of $Q$ for the eigenvalue $1$. 
  Then
  $$h_{\ell}=\hat{h}_{\ell}+ \left\lfloor \frac{\ell}{2}\right\rfloor \xi P\Lambda_0$$
where the  $\hat{h}_{\ell}$ are  bounded terms of an oscillatory character that can be obtained explicitly if desired. Consequently,
  $$ \lim_{\ell\rightarrow \infty}\frac{h_\ell}\ell =\frac12\xi P\Lambda_0$$
 In particular, if $1$ is not in the spectrum of $Q$ (which may be the case in even dimensions), the system has bounded orbits.
 On the other hand, if $\xi$ is not orthogonal to the eigenspace for $Q$ associated to eigenvalue $1$, then generically in the 
 initial conditions orbits are not bounded. In the above formulas, the constant intercollision time has been set to  $1$.
  \end{theorem}
 \begin{proof}
For transversal period $2$ orbits, one has only to consider 
 two values for $\nu_j$, and consequently only two values for $\mathcal{A}(a)$ and a single value  $t_j=t$. 
  From $a_\ell=a_{\ell-1}+ tu_{\ell-1}$ we obtain $h_\ell= h_{\ell-1}+t\sigma_{\ell-1}$. 
   Setting $h_0=0$ and $t_0=1$ without loss of generality, we have 
 $$h_\ell=   \sum_{j=0}^{\ell-1}\sigma_j=  \mathbbm{1}^\dagger \sum_{j=0}^{\ell-1} \Lambda_j= 
 \mathbbm{1}^\dagger \sum_{j=0}^{\ell-1} \mathcal{A}_j\cdots \mathcal{A}_0\Lambda_0. $$
 Here we are setting by convention $\mathcal{A}_0$ to be the identity transformation.  For concreteness, let us assume that $\ell$ is odd:
 $\ell=2m+1$. Then, letting $\mathcal{A}:=\mathcal{A}_1$ and $Q=\mathcal{A}_2\mathcal{A}_1$ gives
 $$h_{2m+1}=  
 \mathbbm{1}^\dagger \left\{\sum_{j=0}^{m} Q^j + \mathcal{A} \sum_{j=0}^{m-1} Q^j\right\}\Lambda_0=
 \mathbbm{1}^\dagger Q \Lambda_0+ \mathbbm{1}^\dagger(I+\mathcal{A}) \sum_{j=0}^{m-1} Q^j\Lambda_0. $$
Notice that 
 $$ \sum_{j=0}^{m-1} Q^j\Lambda_0= m P\Lambda_0+ \sum_{j=0}^{m-1} Q^jP^\perp\Lambda_0.$$
 The summation on the right-hand side of the above equation must be bounded. In fact, further decomposing $P^\perp$ into $2\times 2$ or $1\times 1$ blocks (the latter associated to eigenvalue $-1$ if it is present in the spectrum of $Q$), we end up with sums of a sequence
 of vectors generated by iterating a non-trivial rotation in dimension $2$ or $1$.
 In particular, it follows that if $1$ is not an eigenvalue of $Q$ (in even dimension)  then
  trajectories having transversal period $2$ are necessarily bounded. To conclude, we note that
  $\xi = \mathbbm{1}^\dagger (I +\mathcal{A})$.
 \end{proof}

   \begin{wrapfigure}{L}{0.25\textwidth}
\centering
\includegraphics[width=1.2in]{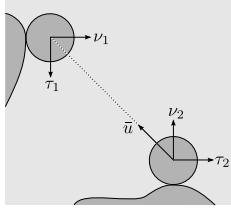}\ \ 
\caption{\label{period2_corner}{\small  Notation  for the  proof of Corollary \ref{drift example}.}}
\end{wrapfigure}

From the above theorem  we can now derive the claim made earlier  that unbounded orbits actually exist, say in dimension $3$, and
obtain the  explicit formula for the longitudinal drift shown 
 in Corollary  \ref{drift example}. This requires that we obtain  the explicit form of the rotation matrix $Q$  and find its
 spectral decomposition. This is an entirely straightforward but somewhat tedious computation, whose details we omit.

Let us introduce the angle
 $\alpha=2\phi$ (see Figure \ref{period2}) and write $c_\alpha=\cos \alpha, s_\alpha=\sin\alpha$.
 Recall that $c$ and $s$ are reserved for the cosine and sine of the special angle $\beta$ determined by the mass distribution parameter $\gamma$. 
Also consider the normal and tangent vectors $\nu_i$ and $\tau_i$ at the two contact points, as indicated in Figure \ref{period2_corner}. Then $Q$ assumes the following form
  $$Q=
    \left(\begin{array}{cc}c^2-s^2 c_\alpha & \left[-sc(1+c_\alpha)\nu_2-ss_\alpha\tau_2\right]^\dagger \\[0.1cm]
 -sc(1+c_\alpha)\nu_1 + ss_\alpha \tau_1& (s^2-c^2c_\alpha)\nu_1\nu_2^\dagger +cs_\alpha( \tau_1\nu_2^\dagger -\nu_1\tau_2^\dagger )-c_\alpha\tau_1\tau_2^\dagger \end{array}\right).$$
  Now observe that 
  $$\eta=\frac1{ \sqrt{s_\alpha^2+2\gamma^2 (1+c_\alpha)} } \left(\begin{array}{c}s_\alpha \\\gamma(\tau_1-\tau_2)\end{array}\right)$$
  is a unit length eigenvector for the eigenvalue $1$ of $Q$. Then an application of the limit formula for the vertical drift from Theorem \ref{drift} gives
  the formula of Corollary \ref{drift example}. Notice that  $P$
  in Theorem \ref{drift} is in this case the rank-$1$ projection on the subspace spanned  by $\eta$.

    \section{Forced billiard motion in a circular cylinder}\label{sec:Forced billiard motion in a circular cylinder}
   In this section we restrict attention to circular cylinders in dimension $n=3$. 
   The main goal is the prove Theorem \ref{main}, restated below after a couple of propositions.

 \begin{proposition}\label{reflection}
If the pre-collision state $(a, u, U)$ of a general (not necessarily a cylinder) no-slip billiard system satisfies the rolling impact condition, then 
the post-collision state is given by
$$C_a(u,U)=\left(u-2 u\cdot \nu_a \nu_a, U\right).$$
In words, the center of mass velocity of the moving particle is reflected specularly and the angular velocity matrix $U$ remains the same.
 \end{proposition}  
 \begin{proof}
 From the definition of the no-slip collision map $(u^+, U^+)=C_a(u^-,U^-)$,  the rolling impact condition
 $ rU^- = u-u\cdot \nu_a\nu_a$, and the relation $c_\beta +\gamma s_\beta = 1$
  we obtain 
 $$
  u^+ = c_\beta u^--\gamma^{-1}{s_\beta} u^-\cdot \nu_a\nu_a + \gamma s_\beta r U^- \nu_a=u^--2 u^-\cdot \nu_a \nu_a
$$
  and 
$$
 U^+=\frac{s_\beta}{\gamma r} \nu_a \wedge u^- + U^- -\frac{s_\beta}{\gamma r} \nu_a\wedge r U^- \nu_a= U^-,
$$
as claimed.
 \end{proof}
   
   Next we restate and prove Proposition \ref{equal times}, which gives a broader context to a property observed in \cite{CFII}. 
\newtheorem*{prop:equal times}{Proposition \ref{equal times}}
\begin{prop:equal times}
  Consider a two-dimensional no-slip billiard system in  a disc. If the first collision satisfies the rolling impact condition, then all subsequent collisions also do, and the times between consecutive  collisions are all equal. Furthermore, the center of mass of the moving particle undergoes specular reflection at each collision.
  \end{prop:equal times} 
  \begin{proof}
Let $a$ and $a'$ be consecutive  collision points on the boundary of $\mathcal{B}$. Let $(u^-,\omega^-)$ denote pre-collision linear and angular velocities at $a$ and $(u^+, \omega^+)$ the post-collision velocities at $a$. Notice that the latter are also the pre-collision velocities
at $a'$. 
 Suppose that the rolling impact condition holds at $a$. Then
as $\omega^-=\omega^+$, we have $$-r\omega^+ = -r\omega^-= u^-\cdot \tau_a = u^+\cdot \tau(a') $$
where the last equality is  due to the post-collision velocity $u^+$ at $a$ being the specular reflection of $u^-$. Therefore the
rolling impact condition also holds at $a'$. That  intercollision times are all equation  is a consequence of
Proposition \ref{reflection}.
  \end{proof}
  
\newtheorem*{thm:main}{Theorem \ref{main}}
\begin{thm:main}
Consider a no-slip billiard system in a circular cylinder in $\mathbb{R}^3$ whose moving particle is subject to a constant force directed along the axis of the cylinder. If the first collision satisfies the transversal rolling impact condition and the first flight segment does not go through the  axis of the cylinder, then the particle's trajectory is bounded. 
 \end{thm:main}
 \begin{proof}
 Reviewing some notation,
 $\mathcal{B}_0$ is here the cylinder of radius $R$ along $e=(0,0,1)^\dagger$  so that $\mathcal{B}$ is the cylinder of radius $R-r$
 along $e$.  
A  trajectory of the billiard system gives a sequence of post-collision states $(a_i, u_i, U_i)\in \mathcal{N}$, $i=0,1, \dots$, and
for this trajectory we  have  the unit normal vectors $\nu_i=\nu(a_i)= -\bar{a}_i/|\bar{a}_i|$ to $\partial \mathcal{B}$ where $\bar{a}=a-a\cdot e e$, the tangent vectors $\tau_i=\tau(a_i)= \nu_i\times e$ to $\partial \mathcal{B}$, the intercollision times 
$t_i$ between the $i$th and $i+1$st collisions, the longitudinal component of the center of mass velocities $\sigma_i= u_i\cdot e$,
the transversal angular velocity vectors $w_i = \gamma r \omega_i\times e = \gamma r U_ie$,
the position  $h_i = a_i\cdot e$ of the center of 
 the  moving particle  along the cylinder's axis, and  $\bar{u}_i=u_i-u_i\cdot e e$. 
  The stating point of the proof are the equations (and notations) recorded in Proposition \ref{longitudinal equations}. We specialize them to this situation by
 noting that the projection $\Pi_a$ appearing in the lower-right block of the matrix $\mathcal{A}(a)$ 
 may be written here as $\tau_a\tau_a^\dagger$.  Thus
 $$\Lambda_i =\left(\begin{array}{c}\sigma_i \\w_i\end{array}\right),  \ \ \mathcal{A}_i= \left(\begin{array}{cc}c & -s \nu_i^\dagger \\-s \nu_i & -c \nu_i \nu_i^\dagger + \tau_i \tau_i^\dagger\end{array}\right), \ \ \mathbbm{1}=\left(\begin{array}{c}1 \\0 \\0\end{array}\right).$$
The  $\Lambda_i$ are vectors in $\mathbb{R}^3$ and the $\mathcal{A}_i$ are $3\times 3$ matrices. 
 Further, the $\mathcal{A}_i$ are orthogonal matrices of determinant $-1$, as is easily checked. 
 With $\Phi=-g\mathbbm{1}$,  then by Proposition \ref{longitudinal equations},
  \begin{equation}
  \Lambda_i = \mathcal{A}_i \left(\Lambda_{i-1} + t_{i-1}\Phi\right).
  \end{equation}
  
  Let $a_0=(R-r,0,0)$ be the initial position of the particle's center of mass, $u_0$ its initial velocity, and $\bar{u}_0=u_0-u_0\cdot e e$ the cross-sectional projection.  We define $\theta$ as the angle between $\bar{u}_0$ and $\tau_0$, as in Figure \ref{initial}, and assume that $0<\theta<\pi/2$. 
  \begin{figure}[htbp]
\begin{center}
\includegraphics[width=2.5 in]{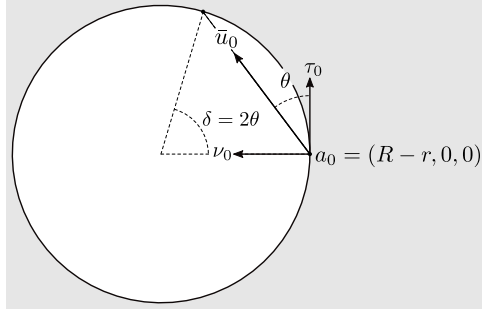}\ \ 
\caption{\small{Cross-sectional projection of initial velocity $\bar{u}_0$ and   definition of $\theta$ and $\delta$.}} 
\label{initial}
\end{center}
\end{figure}

  We also assume without further notice that the transversal rolling impact condition holds. 
 The free-flight times are all equal by Proposition \ref{equal times}; the common value is
 $$t_i = t = \frac{2(R-r)\tan \theta}{|\bar{u}_0|} $$
and $\nu_i= \mathcal{R}(\delta)\nu_{i-1}, \ \  \tau_i= \mathcal{R}(\delta)\tau_{i-1}$
 for $i\geq 1$, where
 $\mathcal{R}(\delta)=\left(\begin{array}{cr}\cos\delta & -\sin\delta  \\\sin\delta & \cos \delta\end{array}\right)$.
Observe that  small values of $\theta$ correspond to near grazing trajectories.

  Define the $3\times 3$ block diagonal matrix $\mathcal{R}=\text{diag}(1, \mathcal{R}(\delta))\in SO(3)$. 
  A simple matrix multiplication shows that $\mathcal{A}_i = \mathcal{R}\mathcal{A}_{i-1}\mathcal{R}^{-1}$ and
  $$\mathcal{A}_0= \left(\begin{array}{rrc}c & -s & 0 \\-s & c & 0 \\0 & 0 & 1\end{array}\right)$$
  where, we recall, $c$ and $s$ are here the cosine and sine of the angle $\beta$. In  terms of mass distribution parameter $\gamma$,  $c=(1-\gamma^2)/(1+\gamma^2)$ and
  $s=2\gamma/(1+\gamma^2)$. All this notation in place, we now have 
  \begin{equation}\label{fundamental}
  h_i= h_{i-1}+ t \mathbbm{1}^\dagger \left\{\Lambda_{i-1}+\frac{t}{2}\Phi\right\}, \ \ 
  \mathcal{A}_i=\mathcal{R}^i\mathcal{A}_0\mathcal{R}^{-i}, \ \ 
  \Lambda_i=\mathcal{A}_i\left\{\Lambda_{i-1}+t\Phi\right\}.
\end{equation}
  We wish to show that the sequence of $h_i$ obtained by  iterating these relations is bounded.
 From Equation (\ref{fundamental}) we obtain
 \begin{equation}\label{hi}
 h_\ell= h_0 - \frac{\ell t^2}{2}g + t\mathbbm{1}^\dagger\left\{\Lambda_0+ \cdots + \Lambda_{\ell-1}\right\}.
 \end{equation} 
  and 
  \begin{equation}\label{Lambdai}
  \Lambda_i= \mathcal{A}_i \cdots \mathcal{A}_1\Lambda_0 +t\left\{\mathcal{A}_i\cdots \mathcal{A}_1+
  \mathcal{A}_i\cdots \mathcal{A}_2+\cdots +\mathcal{A}_i \mathcal{A}_{i-1}+ \mathcal{A}_i \right\}\Phi.
  \end{equation}
 Define $\mathcal{M}=\mathcal{A}_0\mathcal{R}^{-1}$. We also obtain from Equation (\ref{fundamental}), for $j>i$,
 \begin{equation}\label{Si}
 \mathcal{A}_j\cdots \mathcal{A}_{i}=\mathcal{R}^j\mathcal{M}^{j-i+1}\mathcal{R}^{-i+1}.
 \end{equation}
  Equation (\ref{Lambdai}) yields 
  
  \begin{align}
\label{multiple}
\begin{split}
 \Lambda_0+\cdots +\Lambda_{i-1}&= \left\{I + \mathcal{A}_1 + \mathcal{A}_2\mathcal{A}_1+ \cdots + \mathcal{A}_{i-1}\cdots \mathcal{A}_{1}\right\}\Lambda_0\\
 &\ \ \ +t\left\{I\right.\\
 &\ \ \   +\mathcal{A}_1\\
 &\ \ \ +\mathcal{A}_2+ \mathcal{A}_2\mathcal{A}_1\\
 &\ \ \ + \mathcal{A}_3+ \mathcal{A}_3\mathcal{A}_2+ \mathcal{A}_3\mathcal{A}_2\mathcal{A}_1\\
 &\ \ \ \ \ \ \ \ \ \ \ \ \ \ \ \ \ \ \ \cdots\\
 &\ \ \ + \left.\mathcal{A}_{i-1}+ \mathcal{A}_{i-1}\mathcal{A}_{i-2}+\cdots+ \mathcal{A}_{i-1}\cdots \mathcal{A}_{i}\right\}\Phi.
\end{split}
\end{align}
  Since $\mathcal{R}\mathbbm{1}=\mathbbm{1}$, it follows from Equations (\ref{Si}) and (\ref{multiple})
that
\begin{align}\label{LambdaS}
\begin{split}
  \mathbbm{1}^\dagger \left\{\Lambda_0+\cdots+\Lambda_{\ell-1}\right\}&= \mathbbm{1}^\dagger\left\{I + \mathcal{M}+\mathcal{M}^2+\cdots +\mathcal{M}^{\ell-1}\right\}\Lambda_0\\
  &\ \ \ \ \ \ + t\mathbbm{1}^\dagger\left\{(\ell-1)\mathcal{M}+(\ell-2)\mathcal{M}^2+\cdots + 2\mathcal{M}^{\ell-2}+\mathcal{M}^{\ell-1}\right\} \Phi.
  \end{split}
  \end{align}

 It is  now necessary to better understand  $\mathcal{M}$.  This matrix is the product of two orthogonal matrices, hence orthogonal, with determinant $\det \mathcal{M}=(\det \mathcal{A}_0)(\det \mathcal{R}^{-1})=-1$.  It has the explicit form
 $$ \mathcal{M}=\left(\begin{array}{cc}c & -s\nu_1^\dagger \\-s \nu_0 & -c\nu_0\nu_1^\dagger+ \tau_0\tau_1^\dagger\end{array}\right).$$
  Under the assumption $0<\theta<\pi/2$, $\nu_1\neq \pm \nu_0$.   Consider the   orthonormal basis of $\mathbb{R}^3$ defined by the vectors:
  \begin{align*}
  e_0&=\frac{1}{\sqrt{\gamma^2(1+\nu_0\cdot \nu_1)^2+ 2(1+\nu_0\cdot\nu_1)}} \left(\begin{array}{c}\gamma(1+\nu_0\cdot\nu_1) \\\nu_0+\nu_1\end{array}\right)  \\
  e_1&=\frac{1}{\sqrt{2(1-\nu_0\cdot \nu_1)}}\left(\begin{array}{c} 0 \\  \nu_0-\nu_1\end{array}\right)\\
  e_2&=\frac{1}{\sqrt{4 + 2\gamma^2\left(1+\nu_0\cdot\nu_1\right)}}\left(\begin{array}{c}-2 \\\gamma(\nu_0+\nu_1)\end{array}\right)
  \end{align*}
  Then $e_0$ is an eigenvector of $\mathcal{M}$ associated to the eigenvalue $-1$ and the restriction of $\mathcal{M}$ to $e_0^\perp$
     is a planar rotation.  Relative to the basis $\{e_1, e_2\}$, this restriction has matrix form $\left(\begin{array}{cr}a & -b \\b & a\end{array}\right)$ where $a^2+b^2=1$ and
     $a= e_1\cdot\left(\mathcal{M}e_1\right),  b= e_2\cdot\left(\mathcal{M}e_1\right).$  Explicitly,
 \begin{align*}
  a&= 1-\frac{\gamma^2}{1+\gamma^2}\left(1-\nu_0\cdot \nu_1\right)\\
  b&=-\frac{\gamma}{1+\gamma^2}\sqrt{(1-\nu_0\cdot\nu_1)(2+\gamma^2(1+\nu_0\cdot\nu_1))}.
  \end{align*}

  Let $\Pi_-$ and $\Pi_\perp$ denote, respectively, the orthogonal projections from $\mathbb{R}^3$ to the line $\mathbb{R}e_0$ and the plane $e_0^\perp$, and write $\mathcal{M}_\perp$ for the restriction of $\mathcal{M}$ to $e_0^\perp$.  
  Notice that $\mathcal{M}_\perp$ cannot be the identity  (the equation 
  $\mathcal{M}e_1=e_1$ implies $1-\nu_0\cdot\nu_1=0$, which is not the case). 
As  $\mathcal{M}_\perp$ is a planar rotation, $I-\mathcal{M}_\perp$ is nonsingular and
  \begin{align}\label{S-perp}
  \begin{split}
\{\cdots\}_1:=  I + \mathcal{M}+\cdots+\mathcal{M}^{\ell-1}&= \Pi_\perp\left\{ I + \mathcal{M}+\cdots+\mathcal{M}^{\ell-1}\right\}+\Pi_-\left\{ I + \mathcal{M}+\cdots+\mathcal{M}^{\ell-1}\right\}\\
  &= \left\{I+\mathcal{M}_\perp+\cdots +\mathcal{M}_\perp^{\ell-1}\right\}\Pi_\perp+  \{1-1+\cdots +(-1)^{\ell-1}\}\Pi_-\\
  &=(I-\mathcal{M}_\perp)^{-1}\left(I-\mathcal{M}_\perp^\ell\right) \Pi_\perp- \begin{cases} \, 0 &\text{if } \ell =\text{odd}\\ \Pi_-&\text{if } \ell = \text{even.} \end{cases}
  \end{split}
  \end{align}
Notice that $\{\cdots\}_1$ is bounded. Next, consider the expression
$$\{\cdots\}_2:= (\ell-1)\mathcal{M}+(\ell-2)\mathcal{M}^2+\cdots+2\mathcal{M}^{\ell-2}+\mathcal{M}^{\ell-1}. $$
Then 
  \begin{equation}\label{braces2-}
  \Pi_-\{\cdots\}_2= \{-(\ell-1)+(\ell-2)-(\ell-3)+ \cdots + (-1)^{\ell-1}\}\Pi_-\\
  = -\left\lfloor \frac{\ell}{2}\right\rfloor\Pi_-
  \end{equation}
  where $\lfloor\cdot\rfloor$ denotes the floor function. We claim that 
  \begin{equation}\label{braces2perp}
  \Pi_\perp\{\cdots\}_2 = \left\{(\ell-1)(I-\mathcal{M}_\perp)^{-1}\mathcal{M}_\perp
  + (I-\mathcal{M}_\perp)^{-2}\mathcal{M}_\perp^2 \left(I-\mathcal{M}_\perp^{\ell-1}\right)
  \right\}\Pi_\perp
  \end{equation}
  This can be shown to hold as follows. Let $z$ denote a complex variable. Then
  \begin{align*}
  (\ell-1)z+(\ell-2)z^2+\cdots + 2z^{\ell-2}+ z^{\ell-1}&=-z^{\ell+1}\frac{d}{dz}\frac1z\left\{1+\frac1{z}+\cdots +\frac1{z^{\ell-2}}\right\}\\
  &=(\ell-1)\frac{z}{1-z} - \frac{z^2(1-z^{\ell-1})}{(1-z)^2}.
  \end{align*}
  Identifying $\mathbb{R}^2$ with $\mathbb{C}$ and $\mathcal{M}_\perp$ with multiplication by some $z=e^{i\lambda}$ gives the claimed identity.
 Consequently,
 \begin{align}\label{three terms}
 \begin{split}
\mathbbm{1}^\dagger\{\cdots\}_2\Phi&=\left\lfloor\frac{\ell}{2}\right\rfloor  g\mathbbm{1}^\dagger\Pi_-\mathbbm{1} +
-(\ell-1) g \mathbbm{1}^\dagger(I-\mathcal{M}_\perp)^{-1}\mathcal{M}_\perp\Pi_\perp\mathbbm{1}\\
&\ \ \ \ \ \ \ \ \ \  \ \ \ \ \ \ \ \ \ \ \ \ -
g\mathbbm{1}^\dagger
(I-\mathcal{M}_\perp)^{-2}\mathcal{M}_\perp^2\left(I-\mathcal{M}_\perp^{\ell-1}\right)
\Pi_\perp\mathbbm{1}.
\end{split}
 \end{align} 
  The third term on the right-hand side of Equation (\ref{three terms}) is bounded. The first term can be evaluated by noting that
  $$\mathbbm{1}^\dagger\Pi_-\mathbbm{1} = \left(e_0\cdot \mathbbm{1}\right)^2 = \frac{\gamma^2(1+\nu_0\cdot \nu_1)}{2 + \gamma^2(1+\nu_0\cdot\nu_1)}.$$
  Concerning the second term, first observe that $$\Pi_\perp \mathbbm{1}= \mathbbm{1}\cdot e_1 e_1 + \mathbbm{1}\cdot e_2 e_2=
   \mathbbm{1}\cdot e_2 e_2=-\frac{2}{\sqrt{4+2\gamma^2(1+\nu_0\cdot\nu_1)}}e_2$$
   so that 
  $$ \mathbbm{1}^\dagger(I-\mathcal{M}_\perp)^{-1}\mathcal{M}_\perp\Pi_\perp\mathbbm{1}=
  \frac{2}{2+\gamma^2(1+\nu_0\cdot\nu_1)} e_2^\dagger (I-\mathcal{M}_\perp)^{-1}\mathcal{M}_\perp\Pi_\perp e_2. $$ 
 Since the rotation group in dimension $2$ is commutative, the number
 $w^\dagger (I-\mathcal{M}_\perp)^{-1}\mathcal{M}_\perp\Pi_\perp w$ does not depend on the unit vector $w$. Therefore
\begin{align*} e_2^\dagger (I-\mathcal{M}_\perp)^{-1}\mathcal{M}_\perp\Pi_\perp e_2&=
 \left(\begin{array}{c}1 \\0\end{array}\right)^\dagger \left[\left(\begin{array}{cr}a & -b \\b & a\end{array}\right)
 \left(I - \left(\begin{array}{cr}a & -b \\b & a\end{array}\right)\right)^{-1}\right]\left(\begin{array}{c}1 \\0\end{array}\right)\\
 &= \left(\begin{array}{c}1 \\0\end{array}\right)^\dagger \left[
 \frac1{2(1-a)}\left(\begin{array}{cc}a-1 & -b \\b & a-1\end{array}\right)
 \right]\left(\begin{array}{c}1 \\0\end{array}\right)\\
 &=-\frac12.
 \end{align*}
 For concreteness, let us assume  $\ell=$ odd; the case when $\ell$ is even will differ only by a bounded term. For $\ell$ odd we obtain
\begin{equation}\label{braces2Phi} \mathbbm{1}^\dagger \{\cdots\}_2\Phi=   
 +\frac{\ell-1}{2}  g 
-g\mathbbm{1}^\dagger
(I-\mathcal{M}_\perp)^{-2}\mathcal{M}_\perp^2\left(I-\mathcal{M}_\perp^{\ell-1}\right)
\Pi_\perp\mathbbm{1}
\end{equation}
 Returning now to Equation (\ref{hi}),  and using the results so far contained in Equations (\ref{LambdaS}), (\ref{S-perp}) and (\ref{braces2Phi})
  we notice that the unbounded terms $-(\ell t^2 g/2)$ cancel out and we are left with
  $$h_\ell = h_0 + t\mathbbm{1}^\dagger (I-\mathcal{M}_\perp)^{-1}\left(I-\mathcal{M}_\perp^\ell\right)\Pi_\perp \Lambda_0 - {(t^2g/2)}\left\{   
  1 + 2\mathbbm{1}^\dagger(I-\mathcal{M}_\perp)^{-2} \left(\mathcal{M}_\perp^2-\mathcal{M}_\perp^{\ell+1}\right)\Pi_\perp \mathbbm{1}
  \right\}, $$
  which is bonded. This concludes the proof.
  \end{proof}

   \section{Forced motion between parallel planes}\label{sec:Forced motion between parallel planes}
  Here we consider  the billiard domain bounded by 
  two infinite parallel affine codimension-$1$   subspaces of $\mathbb{R}^n$. 
   Let $\nu$ denote the inward-pointing unit normal vector to one of the planes, so that $-\nu$ is the (inward pointing) normal vector for the other plane. Let $e$ be a unit vector perpendicular to $\nu$. We suppose that the billiard particle is subject to a constant force $-gm e$  
 and wish to study the motion of the particle's center of mass along a direction $e$.
  The next theorem, which is a restatement of Theorem \ref{two plates bounded}, asserts that this motion is bounded. 
 \newtheorem*{thm:two plates bounded}{Theorem \ref{two plates bounded}}
\begin{thm:two plates bounded}
 Consider a domain whose boundary consists of two parallel hyperplanes in $\mathbb{R}^n$, $n\geq 2$. Then a trajectory of the no-slip billiard system 
 whose initial center of mass velocity is not parallel to the hyperplanes 
  is bounded. Trajectories remain bounded if a constant force   is applied to the particle's center of mass along   any direction parallel to those hyperplanes.
  \end{thm:two plates bounded}
  \begin{proof}  
  This theorem admits a proof very similar to that of Theorem \ref{main}, but we give instead a more conceptual proof that makes use of a certain invariant quantity  that  we can  identify for the two planes system, but whose possible  counterpart for the circular cylinder is not yet apparent to us. 
  
  For any given set of  initial conditions, the time between two consecutive collisions is constant throughout the orbit;  we denote it by $t$.  As before, we let $a_j$ 
denote the position  of the center of mass of the moving particle at the $j$th collision with the boundary of the billiard domain. 
Due to Theorem \ref{first theorem},  the proof may be approached by induction:  we can focus on the motion in the plane spanned by $e$ and $\nu$ and then argue by induction that trajectories are bounded for the transverse billiard system on $e^\perp$.

Set $\omega_j:= w_j\cdot \nu$, where $w_j:= \gamma  rU_je$ and $U_j$ is the post-collision angular velocity matrix at step $j$.  
 Let the constant force be   $-gm e$, where $m$ is the particle's mass. 
 The component of $a_j$ in the direction $e$ is $h_j=a_j\cdot e$ and the component of the post-collision velocity $u_j$ in the direction  $e$ is
$\sigma_j:= u_j\cdot e$. Then
\begin{equation*}\label{h equation} h_\ell = h_{\ell-1} + t \sigma_{\ell-1} - \frac{t^2g}{2 }=h_0-  \frac{t^2 g}{2}\ell + t \sum_{j=1}^{\ell -1} \sigma_j.\end{equation*} 
It is also useful to introduce the {\em angular displacement} 
$$k_\ell = k_0+t \sum_{j =1}^{\ell -1} \omega_j.$$

The following observation is key: For the billiard domain between two parallel planes, possibly  with a transverse force, the ratio of  angular to linear displacements remains constant after an even number of collisions. In particular,
\begin{equation}\label{slope}\frac{\Delta k}{\Delta h}:=\frac{k_{j+2}-k_j}{h_{j+2}-h_j}=(-1)^j\gamma. \end{equation}
The proof of this claim is a calculation.  For any even $j$ we may reindex to $j=0$, and
$$\frac{\Delta k}{\Delta h}=\frac{t(\omega_0+\omega_1)}{t(\sigma_0+\sigma_1)-t^2g}=\frac{\omega_0+\omega_1}{\sigma_0+\sigma_1-tg}.$$
As before, we write $c_\beta=\cos\beta=\frac{1-\gamma^2}{1+\gamma^2}$ and $s_\beta=\sin\beta=\frac{2\gamma}{1+\gamma^2}.$ Due to Proposition \ref{longitudinal equations},
$$\sigma_0+\sigma_1 = \sigma_0+c_\beta(\sigma_0-tg)+s_\beta \omega_0, \ \ \omega_0+\omega_1 = \omega_0 +s_\beta(\sigma_0-tg) - c_\beta\omega_0.$$
Notice that $1-c_\beta= \gamma s_\beta$ and $1+c_\beta =\gamma^{-1}s_\beta$. Thus we arrive at
 $$ \frac{\Delta k}{\Delta h}=\frac{(1-c_\beta)\omega_0+s_\beta(\sigma_0-tg)}{(1+c_\beta)(\sigma_0-tg)+s_\beta\omega_0}= \gamma.$$
  If $j$ is odd, a similar calculation yields $\frac{\Delta k}{\Delta h}=-\gamma$. 
  It follows from this observation that 
  \begin{equation}\label{lines}k_{2n}-k_0=\gamma(h_{2n}-h_0), \ \ k_{2n+1}-k_1=-\gamma(h_{2n+1}-h_1).\end{equation}
   We will refer to these $(h,k)$ lines as the {\em lines of contact}. 
  The constraint obtained from the existence of the lines of contact, combined with conservation of energy,  bounds the orbits, as we now show.

Notice that the kinetic energy, expressed in terms of $\sigma$ and  the rescaled angular velocity $\omega$  is $\mathcal{K}=\frac12 m\left(\sigma^2+\omega^2\right)$.  Up to a common additive constant, the total energy  at step $j$ is
  $$ \mathcal{E}_j=\frac12m\left(\sigma_j^2+\omega_j^2\right)+mgh_j=E$$
  where $E$ is the constant value of the total energy. 
  Setting $\lambda = t^2g/2$, we have
\begin{equation}\label{hk} \left(h_{2n+1}-h_{2n}\right)^2+ \left(k_{2n+1}-k_{2n}\right)^2=\left( t\sigma_{2n}-\lambda  \right)^2 + \left(t\omega_{2n}\right)^2.\end{equation}
  The  linear relations given by Equations (\ref{lines}) yield
  $$k_{2n+1}-k_{2n}= -\gamma \left(h_{2n+1}+ h_{2n}+ c\right) $$
  where the constant $c$ only depends on  initial values.  The above energy equation  gives
  $$t^2\left(\sigma^2_{2n}+\omega_{2n}^2\right)= \frac{2t^2E}{m}- 4\lambda h_{2n}. $$
  Inserting the previous two equations into (\ref{hk}),
  $$ \left(h_{2n+1}-h_{2n}\right)^2+\gamma^2\left(h_{2n+1}+h_{2n} +c\right)^2+2\lambda\left(h_{2n+1}+h_{2n}\right)=\frac{2t^2E}{m}  +\lambda^2. $$
This is the equation of an ellipse in the $(h_{\text{\tiny odd}}, h_{\text{\tiny even}})$-plane. A similar   ellipse is the locus of  $(h_{\text{\tiny even}}, h_{\text{\tiny odd}})$. Therefore we
 can  conclude that the sequence $h_0, h_1, \dots$ is bounded. 
   \end{proof}
   
  \section{Final comments: chaotic billiards}\label{sec:final comments}
   The examples of no-slip billiards considered so far in this paper (transversal period $2$, parallel hyperplanes, circular cylinder) all 
  share the property that the associated transversal   systems have simple and well-understood behavior. If we were to look for 
  a notion of completely integrable no-slip billiard systems, these would be models to have in mind. 
  
 We wish now to consider   a  numerical example whose transversal dynamics can exhibit chaotic  behavior,  for which
 the problem of bounded orbits is likely to be much more challenging.

Let us revisit  the stadium cylinder, whose cross section was shown in Figure \ref{stadium}. The rolling motion, as already noted, is bounded, and has a typical quasi-periodic  character (see Figure \ref{stadium_curve}),  but the corresponding no-slip billiard behaves much differently. Here we focus on a transition from simple bounded motion to a more chaotic regime at a natural bifurcation point (see Figure \ref{Exp1}) as an illustration of how different these two types of dynamics (namely, rolling motion versus no-slip billiards) can be. 
 To better appreciate the changes to orbits due to changes in initial conditions, it is useful to resort to a visualization device that we have called in \cite{CFZ} a {\em velocity phase portrait.}  We give here a brief review of this simple, but helpful, tool. 

  \begin{figure}[htbp]   
\begin{center}
\includegraphics[width=5in]{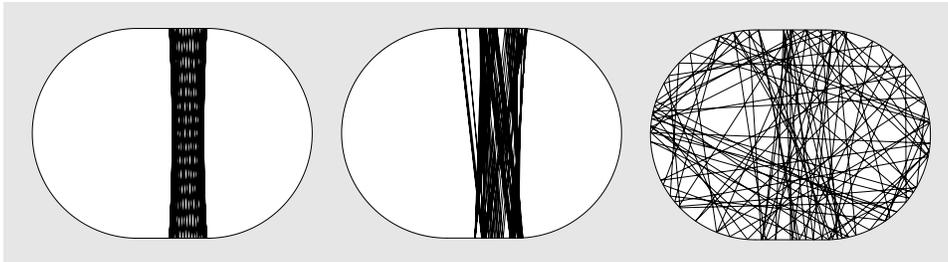}\ \ 
\caption{{\small Transition from regular to chaotic motion. The moving particle begins 
from the middle of the lower flat side with linear velocity pointing up and  a small angular velocity that causes it to move right after the first collision with the upper flat side.  For small values of the angular velocity  
trajectories never touch the curved sides of the boundary, and the motion along the axis of the cylinder is bounded. 
If the initial angular velocity is large enough,  trajectories move beyond the  ends of the flat sides 
and eventually becomes unstable. }}
\label{Exp1}
\end{center}
\end{figure}

For no-slip billiards in dimension $3$, the associated transverse  billiard system has a $3$-dimensional reduced phase space $\mathcal{N}$. (See the definition above  in (\ref{reduced phase space}).)
This space is the product of a $1$-dimensional manifold\----the boundary of the planar billiard domain\----and  a hemisphere in $\mathbb{R}^3$ representing the components of linear and angular velocities relevant to the transversal dynamics.

To make sense of this latter part,
   notice that the two components of the center of mass velocity and the single angular velocity of the planar billiard contribute two degrees of freedom due to conservation of kinetic energy. (The constant force in the vertical direction does not affect the transversal motion due to
Theorem \ref{first theorem}.) Vectors in this hemisphere are most conveniently expressed in the moving frame defined by the unit tangent vector $\tau_a$ to the boundary of the planar billiard domain at a given point $a$, the unit inward pointing normal vector $\nu_a$ at the same point, and a  third unit  vector
perpendicular to the first two representing a unit of angular velocity of the rotating  disc (rescaled by a factor   that turns the kinetic energy into a multiple of the ordinary Euclidean square norm in $\mathbb{R}^3$).
Using this moving frame,   $\nu_a$  at each $a$ is identified with $(0,0,1)$, and  each hemisphere  with the
points of the unit sphere $S^2$ having positive last coordinate. We further project this upper-hemisphere to the unit disc in 
$\mathbb{R}^2$. In this way we  have a bijection between   points in the unit disc and  (linear-angular) $3$-velocities at each boundary point of the planar billiard domain.

 \begin{figure}[htbp]   
\begin{center}
\includegraphics[width=5in]{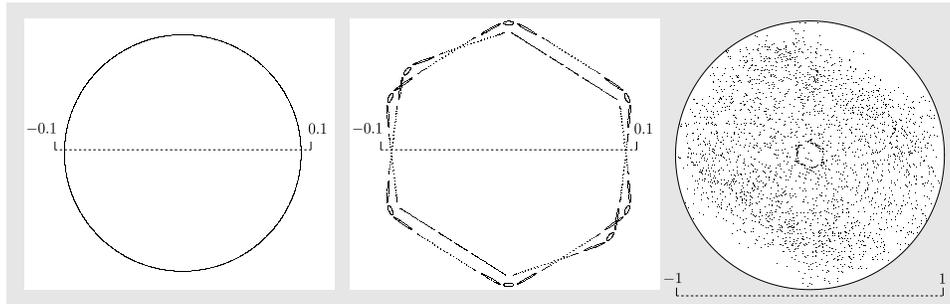}\ \ 
\caption{{\small  Transition from regular to chaotic motion for the transverse dynamics of the stadium cylinder no-slip billiard system, as viewed in the velocity phase portrait. The full velocity space is a disc of radius $1$ as shown in the far right. Initial conditions for the depicted orbits roughly compare to those of Figure \ref{Exp1}.}}
\label{VelPhPor}
\end{center}
\end{figure}

Three transversal orbit segments for the no-slip stadium-cylinder  billiard are shown in Figure \ref{Exp1}. In each, the 
particle begins at the bottom flat side with linear velocity pointing up, and a small angular velocity that causes it to reflect rightward upon first collision. All the other velocity components are set to $0$. When the angular velocity is sufficiently small, 
orbits are confined to the flat parts of the boundary and thus exhibit the bounded motion established in Theorem
\ref{two plates bounded}.

As the angular velocity increases, orbits eventually reach the curved parts, and soon transition to a chaotic regime in which a much larger region of the billiard phase space is explored, as suggested by the rightmost diagram of Figure \ref{Exp1}. 
Figure \ref{VelPhPor} shows what happens during that transition using  the velocity  phase portrait. The regular motion restricted to the flat sides of the boundary has the property that the linear-angular $3$-velocity vector  rotates in a simple fashion, forming  a small circle 
around the north pole of $S^2$. As the trajectory barely crosses into the curved parts of the boundary, the possible linear-angular $3$-velocity
still remains in a small neighborhood of the north pole, but begins to behave in more interesting ways that are very sensitive to the initial velocities. (See the middle diagram in Figure \ref{VelPhPor}.) As the angular velocity increases further, the linear-angular $3$-velocity 
spreads throughout the velocity phase portrait as shown by the rightmost diagram in Figure \ref{VelPhPor}.

The height function accordingly changes from simple bounded behavior (when the motion is limited to the flat boundary parts) 
to the rather more complicated motion over  much wider distances shown in Figure \ref{Zplot0_18}. 
This height function is likely not bounded; in fact,
the graph in Figure \ref{Zplot0_18} suggests a type of ``null-recurrent'' behavior as in one-dimensional  random walks. 
Notice the short periods of fast falling and bouncing back up, separated by
rough plateaux distributed in a seemingly random fashion. We believe that trying to establish limit theorems for the longitudinal motion of chaotic
transverse no-slip billiards in cylinders is a potentially fruitful direction to pursue.

 \begin{figure}[htbp]   
\begin{center}
\includegraphics[width=4.0in]{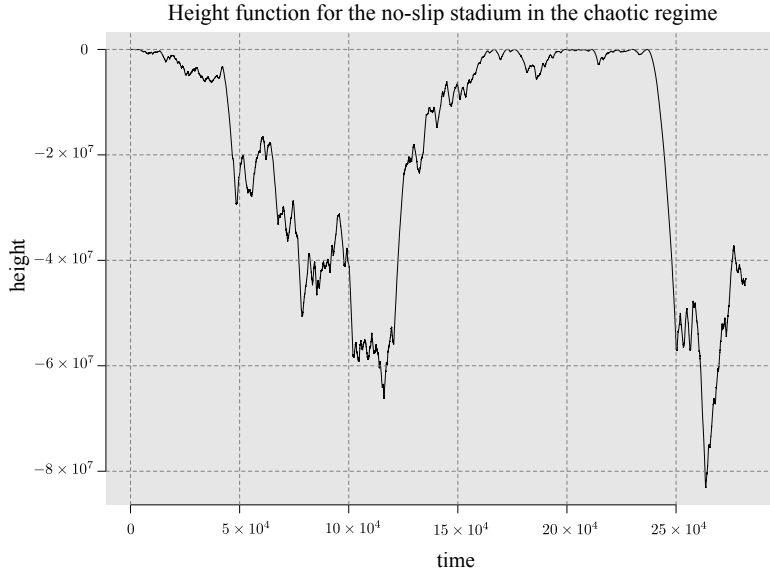}\ \ 
\caption{{\small Height function for the stadium cylinder for a trajectory in the chaotic regime, corresponding to the short orbit segment on the right of Figure \ref{Exp1} and of Figure \ref{VelPhPor}. To give a sense of the scales involved, the diameter of the stadium is $8$ and velocities are of order $1$. The number of time steps is $4\times 10^4$.  }}
\label{Zplot0_18}
\end{center}
\end{figure}

\end{document}